\documentclass[11pt]{amsart}
\usepackage[dvipsnames]{xcolor}
\usepackage[left=0.8in,top=1in,right=0.8in,bottom=1in]{geometry}
\usepackage{wasysym, mathtools, amssymb, tikz, mathrsfs}

\usepackage[T1]{fontenc}
\usepackage[colorlinks=true, linkcolor=orange, urlcolor=violet, citecolor=blue]{hyperref}
\usetikzlibrary{matrix,arrows}

\input xy 
\xyoption{all} %

\usepackage{lineno}

\usepackage{color}

\newtheorem{theorem}{Theorem}[section]
\newtheorem{lemma}[theorem]{Lemma}
\newtheorem{proposition}[theorem]{Proposition}

\newtheorem{cor}[theorem]{Corollary}
\newtheorem{conjecture}[theorem]{Conjecture}

\newtheorem{claim*}{Claim}
\newtheorem{theo}[theorem]{Theorem}
\newtheorem{definition}[theorem]{Definition}

\newcounter{cor}
\newtheorem{thmy}{Theorem}

\newcounter{theo}
\newtheorem{thmyy}{Theorem}

\theoremstyle{definition}
\newtheorem{remark}[theorem]{Remark}

\newtheorem{example}[theorem]{Example}

\renewcommand{\restriction}{\mathord{\upharpoonright}}

\newcommand{\G}{{\mathbb G}}

\newcommand{\PP}{{\mathbb P}}

\newcommand{\Q}{{\mathbb Q}}

\newcommand{\Z}{{\mathbb Z}}

\newcommand{\kc}{\kappa(\mathcal E)}
\newcommand{\hi}{\hat{\imath}}
\newcommand{\hj}{\hat{\jmath}}

\newcommand{\chiSC}{\max\{1, \chi_S(\calC)\}}

\newcommand{\calC}{{\mathcal C}}
\newcommand{\calD}{{\mathcal D}}
\newcommand{\calE}{{\mathcal E}}

\newcommand{\calL}{{\mathcal L}}
\newcommand{\calM}{{\mathcal M}}

\newcommand{\calO}{{\mathcal O}}

\newcommand{\calX}{{\mathcal X}}

\newcommand{\calZ}{{\mathcal Z}}

\usepackage[mathscr]{euscript}

\DeclareMathOperator{\supp}{supp}

\DeclareMathOperator{\ord}{ord}

\DeclareMathOperator{\Pic}{Pic}

\DeclareMathOperator{\Spec}{Spec}

\DeclareMathOperator{\id}{id}

\newcommand{\isom}{\cong}

\numberwithin{equation}{section}
\numberwithin{table}{section}

\newcommand{\Cone}{{\Theta_1}}
\newcommand{\Crs}{{\Theta_2}}
\newcommand{\Ctwo}{{C_1}}
\newcommand{\Cirs}{{C_2}}
\newcommand{\Cfour}{{C_3}}
\newcommand{\Cfive}{C_4}
\newcommand{\Csix}{{C_5}}
\newcommand{\Cseven}{{C_6}}
\newcommand{\Cnine}{{C_7}}
\newcommand{\Cten}{{C_{8}}}
\newcommand{\Cind}{{C_9}}
\newcommand{\Cdep}{{C_{10}}}

\newcommand{\Celeven}{{D_1}}

\newcommand{\Done}{{D_1}}
\newcommand{\Dirs}{{D_2}}
\newcommand{\Dtwo}{{D_3}}
\newcommand{\Dthree}{{D_4}}
\newcommand{\Dfour}{{D_5}}
\newcommand{\Dfive}{{D_6}}
\newcommand{\Dsix}{{D_7}}
\newcommand{\Dseven}{{D_8}}

\title[Lang-Vojta Conjecture over function fields for surfaces dominating $\G_m^2$]{Lang-Vojta Conjecture over function fields \\ for surfaces dominating $\G_m^2$}

\author{Laura Capuano}
\address{DISMA ``Luigi Lagrange'', Politecnico di Torino, Corso Duca degli Abruzzi 24, 10129 Torino, Italy}
\email{laura.capuano@polito.it}

\author{Amos Turchet}

\address{Dipartimento di Matematica e Fisica, Università degli Studi Roma 3, L.go S. L. Murialdo 1, 00146 Roma, Italy}
\email{amos.turchet@uniroma3.it}

\date{}

\begin{document}

\maketitle
 \begin{abstract}
 We prove the nonsplit case of the Lang-Vojta conjecture over function fields for surfaces of log general type that are ramified covers of $\G_m^2$. This extends the results of \cite{CZGm}, where the conjecture was proved in the split case, and the results of \cite{CZConic, TurchetTrans} that were obtained in the case of the complement of a degree four and three component divisor in $\PP^2$. We follow the strategy developed by Corvaja and Zannier and make explicit all the constants involved. 
 \end{abstract}

\section{Introduction}
The celebrated Lang-Vojta conjecture, see \cite[Conjecture F.5.3.6]{HindrySilverman}, predicts degeneracy of $S$-integral points on varieties of log general type over number fields. It is known in full generality for curves, where it reduces to Siegel's theorem (see for example \cite{Siegel}), and for subvarieties of semi-abelian varieties \cite{Falt_ab, vojtaSA, vojtaSA2}. Very deep results have been obtained applying the method developed by Corvaja and Zannier in \cite{CZAnnals}, building on \cite{CZSiegel}, which led to the proof of the conjecture in several new cases, e.g. \cite{CZ_IMRN,CZ_Toh,levin, CZ_Adv,Autissier} (see \cite{CBook} for surveys of known results). 

%

\medskip
In the case of function fields, the Corvaja and Zannier strategy allows one to obtain results that are still out of reach with the current methods in the number field case: for example, in \cite{CZConic} the authors prove the split case of the conjecture for the complement of a conic and two lines in $\PP^2$, a problem which is still open over number fields. The latter result has then been generalized in \cite{CZGm} for isotrivial surfaces that are ramified covers of $\G_m^2$ (see also \cite{Campana, NWY, Lu} for analogous results in the compact and analytic cases).
\medskip

The goal of this article is to prove the non-isotrivial case of \cite{CZGm}. The setting is the following: let $\kappa$ be an algebraically closed field of characteristic 0, let $\kappa(\calC)$ be the function field of a nonsingular projective curve $\calC$ and let $S$ be a finite set of points of $\calC$. Let $(X,D)$ be a pair of log general type over $\kappa(\calC)$, where $X$ is a nonsingular projective surface, $D$ is a simple normal crossings Cartier divisor on $X$ and there is a dominant morphism $X \setminus D \to \G_m^2$ over $\kappa(\calC)$. Let $(\calX,\calD)$ be a log canonical model of $(X,D)$ such that there exists a generically finite dominant morphism $\pi: \calX \to \PP^2 \times \calC$. 

%

\begin{thmy}
 \label{th:main}
 Let $Z$ be the closure of the ramification divisor of $\pi\restriction_{\calX \setminus \calD}$ and assume that the image of the generic fiber $\pi(Z_\eta)$ is disjoint from the singular points of $\PP^2 \setminus \G_m^2$. Then, for every projective embedding $\varphi$ of $\calX$, there exists an explicit positive constant $C = C(Z_\eta,\deg \pi, \varphi)$ such that every section $\sigma: \calC \to \calX$ with $\supp (\sigma^*\calD) \subseteq S$ satisfies
      \[
	\deg \sigma(\calC) \leq C \cdot \max \{ 1, \chi_S(\calC) \},
      \]
      where $\chi_S(\calC)$ is the (negative) Euler characteristic of the affine curve $\calC \setminus S$, i.e. $\chi_S(\calC) = 2 g(\calC) - 2 + \# S$.
\end{thmy}

\noindent The main application of Theorem \ref{th:main} is to complements of normal crossings divisors in $\PP^2_{\kappa(\calC)}$. 

\begin{thmy}
      \label{th:mainP2}
      Let $D$ be a divisor in $\PP^2_{\kappa(\calC)}$ of degree $d \geq 4$ with $r \geq 3$ components. Let $\calD$ be the closure of $D$ in $\PP^2 \times \calC$ and let $S$ be a finite set of points of $\calC$ such that, for every $P \notin S$, the fiber $\calD_P$ has normal crossings singularities. Then, for every projective embedding $\varphi$ of $\PP^2 \times \calC$, there exists a constant $C = C(D,\varphi)$ such that every section $\sigma: \calC \to \PP^2 \times \calC$ with $\supp (\sigma^* \calD) \subseteq S$ verifies
    \[
      \deg \sigma(\calC) \leq C \max\{1, \chi_S(\calC) \}.
    \]
  \end{thmy}
  
The isotrivial case of Theorem \ref{th:mainP2} was proved in \cite[Theorem 1]{CZGm}, and previously in \cite[Theorem 1.1]{CZConic} for $d=4$ and $r=3$; the latter case was obtained in \cite[Theorem 1.3]{TurchetTrans} for non-isotrivial pairs. Note that, when $r \geq 4$, the conclusion is known to hold essentially by a reduction to \cite{Brownawell1986} or \cite{Voloch1985}.

\begin{remark}
 In the statements of Theorem \ref{th:main}
 , given a projective embedding of $\calX$, the degree of $\sigma(\calC)$ is bounded by a constant $C$ multiplied by $\max\{ 1, \chi_S(\calC)\}$, where the constant $C$ depends only on the geometric data of the finite map $\pi$ and its ramification. The dependence on the curve appears \emph{only} in the Euler characteristic. In fact the proof shows that $C$ depends only on the generic fiber $(X,D) \to \G_m^2$ and the projective embedding. This implies that, given a finite cover $\calE \to \calC$ and a pair $(\calX,\calD) \to \calC$ as before, one will obtain the same result for the pair $(\calX, \calD) \times_\calC \calE$ with the \emph{same} constant $C$ (up to the choice of a compatible embedding). This is consistent with more general conjectures of Vojta  (see Section \ref{sec:general_Vojta} or \cite[Section 10.2]{AKnotes} for a more detailed discussion).
\end{remark}

    The main idea in the proof of Theorem \ref{th:main}, as in \cite{CZGm}, is to estimate the contribution of the ramification divisor of the finite map $\pi$ to the height of a section $\sigma \in X(\kappa(\calC))$. More precisely, the strategy of the proof is the following: in Section \ref{sec:depend} we obtain a preliminary result on $S$-units which satisfy a multiplicative dependence relation; this is used in Section \ref{sec:AB}, where we prove an explicit bound for the number of multiple zeros of polynomials in $\kappa(\calC)[X,Y]$ evaluated at $S$-units, extending \cite[Theorem 1.2]{CZConic}. This latter result is the key point to estimate the contribution of the ramification divisor $Z$ to the height of a section, which is obtained in Section \ref{sec:ram}. 
    
    In the nonsplit case one needs to deal with the problem that the log general type assumption does not guarantee in general the positivity of the ramification divisor $Z$. We discuss various results about the divisor $Z$ in Section \ref{sec:ram_pos}. In particular, we show that, even if the divisor $Z$ might not be big, its twist by the pullback of a positive divisor on $\calC$ is big. Moreover, in the case in which the model of the divisor $D$ is ample, the ramification divisor itself can be shown to be big (see Proposition \ref{prop:Zbig}).
    
  We prove Theorem \ref{th:main} in Section \ref{sec:main}, where we apply Proposition \ref{prop:ram} together with the generalized $abc$ inequality over function fields \cite[Theorem B]{Brownawell1986}. Lastly, in Section \ref{sec:appl}, as an application of Theorem \ref{th:main}, we prove Theorem \ref{th:mainP2} and give an explicit example in the case where the base curve is $\PP^1$. 

    
    \subsection{Connections with Vojta's conjectures} \label{sec:general_Vojta}
    A celebrated conjecture in diophantine geometry, proposed by Vojta \cite[Conjecture 3.4.3]{Vojta}, predicts an arithmetic analogue of a (conjectural) higher dimensional Second Main Theorem in Nevanlinna Theory. The conjecture can be extended to describe the distribution of algebraic points on a non-singular projective variety $X$ defined over a number field $k$ and a normal crossings divisor $D$ on $X$. In the stronger form with truncation the conjecture reads as follows:

\begin{conjecture}[See {\cite[Conjecture 24.3]{Vojta2}}]
  \label{conj:LV}
  Let $X,D,k$ as above. Let $S$ be a finite set of places of $k$ containing the Archimedean ones, let $K_X$ be the canonical divisor of $X$, let $A$ be an ample divisor on $X$ and let $r$ be a positive integer. Then, for every $\varepsilon > 0$, there exists a proper Zariski closed subset $W = W(k,S,X,D, A,\varepsilon,r)$ of $X$ such that the inequality
  \begin{equation*}
     h_{K_X + D}(x) -\varepsilon h_A(x) \leq d_k(x) + N_S^{(1)}(D,x) + O(1)
  \end{equation*}
  holds for almost all $x \in (X \setminus W)\left (\overline{k} \right)$ with $[k(x) : k] \leq r$.
\end{conjecture}

In the statement of the conjecture, $d_k(x)$ denotes the \emph{logarithmic discriminant} of the point $x$ (see \cite[Definition 23.1]{Vojta2}) and $N_S^{(1)}(D,x)$ denotes the \emph{truncated counting function} (see \cite[Definition 22.4]{Vojta2}). Conjecture \ref{conj:LV} has a wide range of important consequences; we mention for example the \emph{abc} conjecture of Masser-Oesterl\'e \cite[Conjecture 3]{abc}, the Bombieri-Lang conjecture \cite[Conjecture F.5.2.1]{HindrySilverman} and the Lang-Vojta conjecture \cite[Conjecture F.5.3.6]{HindrySilverman}. We refer to \cite{ABT,AdVT,AKnotes,Jnotes} for further consequences.

\medskip

In this paper we deal with the function field case of Conjecture \ref{conj:LV}, i.e. when the number field $k$ is replaced by the function field $\kappa(\calC)$ of a non-singular projective curve over an algebraically closed field $\kappa$ of characteristic zero. We point out that Theorem \ref{th:main} is related to the function field analogue of Conjecture \ref{conj:LV}. To see this, consider $(X,D)$ as in Conjecture \ref{conj:LV}: points $x \in X(k)$ correspond over function fields to sections $\sigma: \calC \to \calX$, for a projective model $\calX$ of the variety $X$. Similarly, the height bound corresponds to a degree bound for the image $\sigma(\calC)$, while the contributions of the discriminant and the truncated counting function correspond to the Euler characteristic $2g(\calC) - 2 + \# S$.

We also note that, in Theorem \ref{th:main} 
the exceptional set $W$ does not appear explicitly. Moreover, our result shows that a degree bound holds also for sections for which a bound as in Conjecture \ref{conj:LV} does not hold. On the other hand, our theorem contains a constant $C$ that does not appear in Vojta's formulation. For further discussion on the exceptional set we refer to \cite[Section 3]{RTW}.
\medskip
 

The reader can find a detailed analysis of the case $X = \PP^\ell$, $\calC = \PP^1$ and $S = \{0, \infty \}$ in \cite[Section 5]{CLZ19}, where the Nevanlinna analogue is also discussed. Moreover in \cite{CLZ19} the authors obtain new cases of Conjecture \ref{conj:LV} for rational points in higher dimensions, adopting a function field version of the method introduced in \cite{Levin2018} (which in turn extended \cite{CZ05}). These results can be seen as higher dimensional cases of \cite[Theorem 2]{CZGm}, therefore we expect that the methods of the present paper can be further generalized to give higher dimensional analogues of Theorem \ref{th:main}.

    \subsection*{Acknowledgements}
    We thank Kenny Ascher, Lucas Braune, Pietro Corvaja, Kristin DeVleming, Carlo Gasbarri, Ariyan Javanpeykar, S\'andor Kov\'acs, Aaron Levin, Siddarth Mathur and Julie Wang for useful conversations. This paper was partly written during visits of the two authors to the Department of Mathematics of University of Washington, the Mathematical Institute of University of Oxford and the Politecnico of Torino: we thank all the institutions for providing an excellent working environment. Research of Capuano was partly supported by funds from EPSRC EP/N008359/1. Research of Turchet was supported in part by funds from NSF grant DMS-1553459 and by Centro di Ricerca Matematica Ennio de Giorgi. Both authors are members of the GNSAGA group of INDAM.

    \section{Setting and Notations}

    \subsection{Function fields.} In this paper we will denote by $\calC$ a nonsingular projective curve (integral, separated scheme of finite type of dimension 1) defined over an algebraically closed field $\kappa$ of characteristic zero and by $S$ a finite set of (closed) points of $\calC$. We will denote by $\calO_S$ the ring of $S$-integers, i.e. the ring $\kappa[\calC \setminus S]$ of regular functions in the complement of $S$: its elements are rational functions on $\calC$ with poles contained in $S$. Similarly, we will denote by $\calO_S^*$ the group of $S$-units, i.e. the group of invertible elements of $\calO_S$: its elements are rational functions on $\calC$ with both zeros and poles contained in $S$. If $g(\calC)$ is the genus of the curve $\calC$, then the (negative) Euler characteristic of the affine curve $\calC \setminus S$, denoted by $\chi_S(\calC)$, is defined as
    \[
      \chi_S(\calC) := \chi(\calC \setminus S) = 2 g(\calC) - 2 + \# S.
    \]
Notice that if we have at least a nonconstant $S$-unit, then the cardinality of $S$ is at least $2$ hence $\chi_{S}(\calC)$ will always be non negative.

    For any rational function $a \in \kappa(\calC)$ we denote by $H_\calC(a)$, or simply by $H(a)$ when the reference to the curve is clear, the height of $a$, i.e. its degree as a morphism to $\PP^1$. This is equivalent to the usual definition of Weil Height via valuations as follows: every point $P \in \calC$ induces a discrete valuation of the field $\kappa(\calC)$, trivial on $\kappa$, that can be normalized such that its value group is $\Z$. We denote by $\ord_P$ the corresponding valuation on $\kappa(\calC)$. Then, the height of a function $a \in \kappa(\calC)$ can be expressed as
    \[
      H_\calC(a) = \sum_{P \in \calC} \max \{ 0, \ord_P(a) \}=-\sum_{P \in \calC} \min \{ 0, \ord_P(a) \} .
    \]
    If $\calE$ is a nonsingular projective curve and $\calE \to \calC$ is a dominant morphism (corresponding to an inclusion $\kappa(\calC) \subseteq \kappa(\calE)$), then the heights of a rational function $a \in \kappa(\calC)$, with respect to $\calC$ and $\calE$, verify
\begin{equation*}
      H_\calE(a) = [\kappa(\calE):\kappa(\calC)] H_\calC(a).
\end{equation*}
For $n\ge 2$ and elements $a_1, \ldots, a_n \in \kappa(\calC)$, we denote by $H(a_1: \cdots : a_n)$ the projective height
\begin{equation*}
H(a_1: \cdots :a_n)=-\sum_{P\in \calC} \min\{\ord_P(a_1), \ldots, \ord_P(a_n)\}.
\end{equation*}    
Given a polynomial $F \in \kappa(\calC)[X_1,\dots,X_n]$, the height of $F$, denoted by $H_{\calC}(F)$, will always be the maximum of the heights of its coefficients. For more details about heights we refer to \cite{Vojta2}. \\
    
An important tool over function fields is the presence of derivations. Following \cite{CZConic}, we will fix a differential form on $\calC$ in order to define the ``derivative'' of a rational function using the following lemma.

    \begin{lemma}[{\cite[Lemma 3.5]{CZConic}}] \label{lem:deriv} 
       Given a nonsingular projective curve $\calC$ of genus $g$ and a finite set of points $S \subset \calC$, there exist a meromorphic differential form $\omega \in \calC$ and a finite set $T \subset \calC$ such that $\#T = \max \{ 0, 2g -2 \}$ and, for every $u \in \calO_S^*$, there exists an $(S \cup T)$-integer $ \theta_{u} \in \calO_{S \cup T}$ having only simple poles such that
\begin{equation*}
  \frac{d(u)}{u} = \theta_{u} \cdot \omega \qquad \text{ and } \qquad H_{\calC}( \theta_{u}) \leq \chi_{S}(\calC).
\end{equation*}
    \end{lemma}
    In the rest of the paper the form $\omega$ will be fixed (compatibly for every finite cover $\calE \to \calC$) and, for $a \in \kappa(\calC)$, we will denote by $a'$ the rational function that satisfies $d(a) = a' \cdot \omega$. With this notation, the rational function $\theta_u$ appearing in the previous lemma is equal to $u'/u$.
 
    \subsection{Surfaces over function fields and fibered threefolds}
    The main focus of this paper is non-isotrivial surfaces defined over function fields and their models. We recall here the main definitions, fixing notations and terminology.
    
    \begin{definition}
      \label{def:model}
      Given a projective variety $X$ of dimension $n$ defined over the function field $\kappa(\calC)$, a (proper) \emph{model} $\calX$ of $X$ over $\calC$ (or over $\kappa(\calC)$) is the datum of a proper flat map $\rho: \calX \to \calC$ and an isomorphism $\calX \times_\calC \Spec \kappa(\calC) \isom X$.
  \end{definition}

From this definition it follows that the model of a surface $X$ over the function field $\kappa(\calC)$ is a fibered threefold $\calX \to \calC$. 
We note that, in Definition \ref{def:model}, the model $\calX$ can be singular. In this paper we will always restrict to the case in which the total space $\calX$ has only mild singularities: in particular, we will consider only models of nonsingular surfaces that are normal.
    
When dealing with an affine variety $Y$, we will identify it with a pair $(X,D)$, where $X$ is a projective variety, $D$ is a normal crossings divisor and $X \setminus D \cong Y$. Even if this identification is not unique, in this paper we will use the language of pairs, since it is more natural from the geometric point of view. Moreover, when $Y$ is a nonsingular affine surface, one can always consider a canonical choice for $(X,D)$, namely a minimal log resolution. In this latter case, the pair $(X,D)$ can be chosen to be \emph{log smooth}, i.e. $X$ is nonsingular and $D$ has simple normal crossings singularities.

\begin{definition}
  Given a pair $(X,D)$, a \emph{model} of $(X,D)$ is a model $\rho: \calX \to \calC$ of the projective variety $X$ over $\calC$ together with a model of $D$ whose total space is a Cartier divisor $\calD$ of $\calX$. We view the model as a family of pairs $(\calX,\calD) \to \calC$. Given an integral affine variety $Y$, a model of $Y$ is a model of the corresponding pair $(X,D)$.
\end{definition}
    
    Similarly as before we will restrict to the case in which the model of a log smooth pair $(X,D)$ has only mild singularities: in particular, we will consider models of a log smooth pair that have log canonical singularities. This still implies that the total space of $\calX$ is normal, but takes into account the presence of the divisor $D$. We refer to \cite[Chapter 2]{singmmp} for the precise definition and properties of log canonical singularities.

    In the setting of Conjecture \ref{conj:LV}, we are interested in affine surfaces $Y$ of log general type.
    
    \begin{definition}
      An affine variety $Y$ is of \emph{log general type} if for a(ny) log resolution $(\widetilde{Y},E)$ of $Y$, the log canonical divisor $K_{\widetilde{Y}} + E$ is big. This property is independent of the choice of the log resolution. If we identify $Y$ with $(X,D)$, we say that the pair is of log general type if $X \setminus D \cong Y$ is of log general type.
    \end{definition}

     We note that being of log general type does not extend naturally to models. Indeed, if $(\calX,\calD) \to \calC$ is a model of a pair $(X,D)$ over the function field $\kappa(\calC)$, the fact that $(X,D)$ is of log general type does not imply in general that the log canonical divisor of the total space $(\calX,\calD)$ is big.\medskip

     Finally we define non-isotrivial pairs: these are pairs that cannot be trivialized after a finite base change, i.e. there exists no finite base change $\calE \to \calC$ such that the base changed pair is isomorphic to a product $(X',D') \times_{\Spec \kappa} \Spec \kappa(\calE)$, for a pair $(X',D')$ defined over $\kappa$.

     In the case in which the fibers of the model $(\calX,\calD)$ have ample log canonical and mild singularities, being non-isotrivial is equivalent to require that the the moduli map $\calC \to \overline{\calM}_h$ to the KSBA moduli space of stable pairs $\overline{\calM}_h$ is not constant. 

    \subsection{Threefolds dominating $\PP^2 \times \calC$}\label{sec:3folds}
    In this article we consider models of non-isotrivial pairs $(X,D)$ over the function field $\kappa(\calC)$ where $X \setminus D$ is a ramified cover of $\G_m^2$. These correspond to fibrations of the form $\rho: (\calX,\calD) \to \calC$, together with dominant maps $\pi: \calX \to \PP^2 \times \calC$, that restrict to finite maps in the complement of $\calD$. 

    We note that in practice the threefold $\calX$ will be given by a dominant \emph{rational} map \linebreak $\pi: \calX \dashrightarrow \PP^2 \times \calC$ whose indeterminacy locus is contained in $\calD$. Moreover, the irreducible components of the indeterminacy locus have dimension at most 1, and their images under the map $\rho$ are finite sets of points of $\calC$ (and contained in the set $S$). Resolving the indeterminacy of $\pi$ gives a threefold $\calX'$ and a morphism $\pi': \calX' \to \PP^2 \times \calC$ that coincides with $\pi$ in the complement of $\calD$. Moreover, every section $\sigma: \calC \to \calX$ extend to a section $\sigma': \calC \to \calX'$. Therefore we can always assume that $\pi$ is a morphism, up to resolving the indeterminacy locus and replacing it with $\pi'$, since this will affect neither the map in the complement of $\calD$ nor the sections that we will consider.
    \medskip
    
\noindent We denote by $Z$ the closure of the ramification divisor of the restriction $\pi\restriction_{\calX\setminus \calD}$. We list here the conditions that will be assumed in the rest of the paper.
 \begin{itemize}
 \item $(X,D)$ is a log smooth non-isotrivial pair of log general type over the function field $\kappa(\calC)$ such that $X\setminus D$ is a ramified cover of $\G_m^2$; 
 \item $\rho: (\calX,\calD) \to \calC$ is a log canonical model of $(X,D)$, in particular $\calX$ is normal;
   \item the fibers of $\calD \to \calC$ have simple normal crossings singularities outside of $S$;
   \item the map $\pi$ is compatible with the fibration, i.e. $pr_2 \circ \pi = \rho$, where $pr_2 : \PP^2 \times \calC \to \calC$ is the second projection;
   \item on every fiber of $\rho$ outside of $S$, the map $\pi$ restricts to a finite dominant map to $\G_m^2$ in the complement of $\calD$ and the restriction of the divisor $\calD$ is the pull-back of the boundary divisor $\PP^2 \setminus \G_m^2$;
   \item on every fiber of $\rho$ outside of $S$, the image of the ramification divisor $Z$ of the restriction of $\pi$ avoids the singular points of the boundary of $\G_m^2$.
 \end{itemize}
The situation is made explicit in the following diagram:
    \[
      \xymatrix{ (\calX,\calD) \ar[d]_\rho \ar[r]^\pi & \PP^2 \times \calC \ar[dl]^{pr_2} \\ \calC \ar@/^2.0pc/[u]^\sigma & }
    \]

%
    In the case in which the generic fiber $(X,D)$ of $\rho$ is of log general type, given a projective embedding $\varphi$ of $\calX$, one expects the existence of a constant $C = C(X,D,\pi,\varphi)$ such that for every section $\sigma: \calC\setminus S \to \calX\setminus \calD$ one has
    \[
      \deg_\varphi(\sigma(\calC)) \leq C \max\{1, \chi_S(\calC) \},
    \]
    where $\deg_\varphi$ denotes the corresponding degree in the projective space where $\calX$ is embedded. Note that different embeddings give rise to different constants but the existence of the bound is independent of the choice of the embedding.  For this reason we will drop the explicit dependence on $\varphi$ and we will assume that all bounds depend on the choice of the embedding.\medskip

\section{Multiplicative Dependence between $S$-units}
\label{sec:depend}

In order to prove Theorem \ref{th:main} we obtain height bounds that yield the degree bound predicted by the conjecture. Recall that, in the setting of Section \ref{sec:3folds}, a section $\sigma: \calC \setminus S \to \calX \setminus \calD$, or equivalently a section $\sigma: \calC \to \calX$ such that $\supp (\sigma^*\calD) \subseteq S$, corresponds to a $D$-integral point in $X(\calO_S) = X(\kappa[\calC \setminus S])$; similarly, a map $\pi \circ \sigma: \calC \to \PP^2 \times \calC$ such that $pr_2 \circ (\pi \circ \sigma) = \id_\calC$ corresponds to a $S$-unit point $(u,v) \in \G_m^2(\calO_S^*)$, where $u,v$ are rational functions on $\calC$ with zeros and poles contained in $S$. The heights of $u$ and $v$ and the degree of $\sigma(\calC)$ are strictly related, since a bound on the heights of $u$ and $v$ gives a bound on the degree of the image $\sigma(\calC)$. On the other hand, a bound on the degree of the image does not guarantee that the heights of $u$ and $v$ are bounded, since, for example $u$ and $v$ might be multiplicatively dependent (we refer to \cite[Section 2]{CZGm} for a detailed discussion). We focus on this latter case in this section. 
\medskip

 Let $A(X,Y) \in \kappa(\mathcal C)[X,Y]$ be an irreducible polynomial of the form
\[
  A(X,Y)= \sum_{i+j \leq \deg A} \lambda_{ij} X^i Y^j, 
\]
where $\deg A := \deg_X A + \deg_Y A$. Let $u_1, u_2 \in \kappa(\calC)$ be nonzero rational functions and let $B(X,Y)\in \kappa(\calC)[X,Y]$ be defined, in terms of $A(X,Y), u_1, u_2$ as
\[
B(X,Y)= \frac{u_1'}{u_1} X \frac{\partial}{\partial X} A(X,Y) + \frac{u_2'}{u_2} Y \frac{\partial}{\partial Y} A(X,Y)+ \sum_{i+j \leq \deg A} \lambda_{ij}' X^iY^j. 
\]
Note that the derivative of $A(u_1,u_2)$ coincide with $B(u_1,u_2)$.\medskip

In this section we derive a dependence relation for $S$-units $u_1$ and $u_2$ assuming they satisfy a relation of the form $(u_1/\alpha)^r(u_2/\beta)^s = \mu$ for some constant $\mu \in \kappa^\times$ and $\alpha,\beta$ roots of $A$ and $B$. In general, we cannot expect $u_1$ and $u_2$ to satisfy the conclusion of \cite[Lemma 3.14]{CZConic}, i.e. $u_1^r u_2^s = \mu'$, for some $\mu' \in \kappa^\times$; instead we prove that there exists a fixed $S$-unit $\gamma$, independent of $u_1$ and $u_2$, such that $u_1^r u_2^s = \gamma$. This will be sufficient for the applications in this paper.

\begin{lemma} \label{lem:314}
	Let $A$ and $B$ as before and let $(\alpha, \beta)$ be a common zero of $A$ and $B$ in $\kappa(\calC)$. If $u_1 /\alpha$ and $u_2 /\beta$ satisfy a multiplicative dependence relation of the form
\begin{equation} \label{m.relation}
  \Big( \frac{u_1}{\alpha} \Big)^r \Big( \frac{u_2}{\beta} \Big)^s = \mu,
\end{equation} 
for a suitable pair of nonzero integers $(r,s)\in \Z^2$ and a constant $\mu \in \kappa^\times$, then either one between $u_1/\alpha$ and $u_2/\beta$ is constant or there exists $\gamma \in \overline{\kappa(\lambda_{ij})}$ independent of $u_1, u_2$ such that $u_1^r u_2^s = \gamma$.
\end{lemma}

\begin{proof} 
Assume that $u_1/\alpha$ and $u_2/\beta$ satisfy a multiplicative dependence relation of the form \eqref{m.relation}. If $(\alpha, \beta)$ is a singular point of $A(X,Y)$, then it is defined in $\overline{\kappa(\lambda_{ij})}$ independently of $u_1$ and $u_2$; so we can conclude with $\gamma = \mu \alpha^r \beta^s$. Therefore, we can assume that at least one between $\frac{\partial}{\partial X} A(\alpha, \beta)$ and $\frac{\partial}{\partial Y} A(\alpha, \beta)$ is non zero.

Notice moreover that, without loss of generality, we can assume that $r, s$ are coprime integers. If they are not, then we can write $r=r' d$ and $s=s' d$ with $(r', s')=1$; then, \eqref{m.relation} would imply that there exists $\mu' \in \kappa$ such that $\big(  \frac{u_1}{\alpha} \big)^{r'} \big(  \frac{u_2}{\beta} \big)^{s'}= \mu'$.
\medskip

Let us define 
\[
  \Gamma(X,Y):= \sum_{i+j\leq \deg A} \lambda_{ij}' X^i Y^j.
\]
Since $A(\alpha, \beta)=0$, taking differentials, we obtain
\begin{equation} \label{A'}
  \alpha' \frac{\partial}{\partial X} A(\alpha, \beta)+ \beta' \frac{\partial}{\partial Y} A(\alpha, \beta)+ \Gamma(\alpha, \beta)=0.
\end{equation} 
By definition of $\alpha$ and $\beta$, we have that $B(\alpha, \beta)=0$, i.e.
\begin{equation} \label{B}
  \frac{u_1'}{u_1} \alpha \frac{\partial}{\partial X} A(\alpha, \beta) + \frac{u_2'}{u_2} \beta \frac{\partial}{\partial Y} A(\alpha, \beta) + \Gamma (\alpha, \beta)=0.
\end{equation}
By \eqref{m.relation} we have a linear relation of the form 
\begin{equation} \label{lrel}
   r \left ( \frac{u_1'}{u_1} - \frac{\alpha'}{\alpha} \right ) + s \left ( \frac{u_2'}{u_2} - \frac{\beta'}{\beta} \right )=0.
\end{equation}
Taking the difference between \eqref{B} and \eqref{A'} and multiplying \eqref{lrel} by $\alpha \beta$, we obtain the system:
\begin{equation} \label{system}
  \begin{cases}
       \left ( \frac{u_1'}{u_1} \alpha - {\alpha'} \right ) \frac{\partial}{\partial X} A(\alpha, \beta) + \left ( \frac{u_2'}{u_2} \beta - {\beta'} \right ) \frac{\partial}{\partial Y} A(\alpha, \beta)=0 \\
       \left ( \frac{u_1'}{u_1} \alpha - {\alpha'} \right ) r \beta +  \left ( \frac{u_2'}{u_2} \beta - {\beta'} \right ) s \alpha=0 
  \end{cases}.
\end{equation} 
From \eqref{system} it is easy to see that either $u_1'/u_1 = \alpha'/\alpha$ and $u_2'/u_2 = \beta'/\beta$, which implies that both $u_1/\alpha$ and $u_2/\beta$ are constant, or we have
\begin{equation*}
  s\alpha \frac{\partial}{\partial X} A(\alpha, \beta) - r\beta \frac{\partial}{\partial Y} A(\alpha, \beta)=0 .
\end{equation*}
We define $ A^*(X,Y):=sX \frac{\partial}{\partial X} A(X, Y) - rY \frac{\partial}{\partial Y} A(X,Y)$; as by assumption $A(X,Y)$ is irreducible and $\deg A^* \leq \deg A$, we can have either that $({A}^*(X,Y), A(X,Y))=1$ or, there exists a constant $a \in \kappa$ such that ${A}^*(X,Y)=a A(X,Y)$. Let us analyze the two cases separately.

Suppose first that $({A}^*(X,Y), A(X,Y))=1$. We have that $(\alpha, \beta)$ is a common zero of the polynomials ${A}^*(X,Y)$ and $A(X,Y)$. As the two polynomials are coprime, Bezout's theorem ensures that the number of common solutions is finite and bounded by $(\deg A)^2$. Hence, $\alpha$ and $\beta$ are two rational functions in $\kappa(\calC)$ independent of $u_1$ and $u_2$, and \eqref{m.relation} can be rewritten as
\begin{equation*}
 u_1^r u_2^s= \mu {\alpha}^r {\beta}^s. 
\end{equation*}
Therefore $\gamma= \mu \alpha^r \beta^s $ is an algebraic function of the coefficients of $A$ as wanted.

Let us finally see what happens if ${A}^*(X,Y)=a A(X,Y)$ for some $a \in \kappa$. If $a = 0$, i.e. $A^*(X,Y)$ is identically zero, then the polynomial $A$ is of the form $\lambda_{hr\, hs}(X^r Y^s)^h + \lambda_{00}$ for a non-zero integer $h$. This implies that $\gamma = \mu \alpha^r \beta^s$ is independent of $u_1,u_2$ and $u_1^r u_2^s = \gamma$ as wanted.

On the other hand, if $a \neq 0$, one gets $\lambda_{00}=0$ and, for all $ i,j$ such that $1 \le i+j \le \deg A$ and $\lambda_{ij}\neq 0$, we have 
\begin{equation} \label{eq:rel}
si-rj=a.
\end{equation}
Moreover, as by assumption $A$ is irreducible, $A$ has both a monomial that contains only $X$ and a monomial that contains only $Y$, which implies that $a$ is a non-zero integer divisible by $rs$ (as we are assuming $r$ and $s$ coprime), $sa>0$ and $ra<0$.

Let us assume that $a>0$ (the other case is completely symmetric); then, we must have $s>0$ and $r<0$. From \eqref{eq:rel} and the fact that $a$ is divisible by $r$ and $s$, the polynomial $A$ will be of the form 
\begin{equation} \label{eq:A_form}
A(X,Y)=\sum_{j=0}^{m} \lambda_j X^{-r(m-j)} Y^{sj},
\end{equation}
for some $m>0$. As $A(\alpha, \beta)=0$, from \eqref{eq:A_form} we have
\[
\sum_{j=0} \lambda_j \left (\alpha^r \beta^s \right )^j=0,
\]
therefore $\gamma = \mu \alpha^r \beta^s$ is independent of $u_1, u_2$, and $u_1^r u_2^s = \gamma$. 
\end{proof}

\begin{remark}
We point out that, as $(\alpha, \beta)$ is a common zero of $A$ and $B$ where $B$ depends on $u_1$ and $u_2$ and their derivatives, writing $u_1^r u_2^s= \mu \alpha^r \beta^s$ does not directly give the desired conclusion. In the proof of the previous Lemma we however obtained that either $u_1/\alpha$ and $u_2/\beta$ are constant, or the quantity $\mu\alpha^r \beta^s$ is an algebraic function of the coefficients of $A$, and so it is independent of $u_1$ and $u_2$. In particular, if $u_1/\alpha$ and $u_2/\beta$ are not constant, given $r$, $s$ and $\mu$ there are only finitely many $\gamma \in \kc$ such that $u_1^r u_2^s=\gamma$.
\end{remark}

%
%
%
%
%

\section{Counting multiple zeros}
\label{sec:AB}
The goal of this section is to prove a bound for the number of multiple zeros of polynomials evaluated at $S$-units; this extends explicitly \cite[Theorem 1.2]{CZConic} to polynomials with nonconstant coefficients. Moreover, we give an explicit bound on the exponents of a multiplicative relation between the $S$-units, when the bound on the number of multiple zeros might not hold. 

Let $\calC$ be a smooth projective curve and let $S \subset \calC$ be a finite set of points. Let $A(X,Y) \in \kappa(\calC)[X,Y]$ be a polynomial without repeated factors. We write the polynomial as
\[
  A(X,Y) = \sum_{i+j \le \deg A} \lambda_{ij} X^i Y^j,
\]
where $\deg A = \deg_X A + \deg_Y A$ as before. Recall that we denote by $H_{\calC}(A)$ the height of the polynomial $A$, which is defined as the maximum of the heights of its coefficients. Then, we have the following result.

 \begin{theorem}
   \label{th:1.2}
   Let $\varepsilon > 0$ be a positive real number. Then, there exist constants $\Cone = \Cone(\deg A, H_{\calC}(A),$ $\varepsilon)$ and $\Crs = \Crs(\deg A, \varepsilon )$ such that, for all pairs $(u_1,u_2) \in (\calO_S^*)^2$ with
   \[
     \max \{ H_{\calC}(u_1), H_{\calC}(u_2)\} \geq \Cone \max \{ 1, \chi_S(\calC) \},
   \]
at least one of the following holds:
   \begin{itemize}
	   \item the $S$-units $u_1,u_2$ verify a relation of the form $u_1^r u_2^s = \gamma$ for a pair of integers $(r,s) \in \Z^2 \setminus \{(0,0)\}$ such that $\max \{ \lvert r \rvert, \lvert s \rvert \} \leq \Crs$ and $\gamma$ is an algebraic function of the coefficients of $A$ independent of $u_1$ and $u_2$;
	   \item the rational function $A(u_1,u_2)$ verifies
   \begin{equation}
     \sum_{P \in \calC \setminus S} \max \{ 0, \ord_P(A(u_1,u_2)) -1 \} \le \varepsilon \max\{ H_{\calC}(u_1),H_{\calC}(u_2) \}.
     \label{eq:th12}
   \end{equation}
   \end{itemize}
 \end{theorem}

\begin{proof}[Proof of Theorem \ref{th:1.2}]

We factor $A(X,Y)$ in irreducible polynomials in $\kappa(\calC)[X,Y]$ as
 \begin{equation} \label{eq:factorization}
   A(X,Y) = A_1(X,Y)A_2(X,Y)\cdots A_l(X,Y).
 \end{equation}

 We begin by noticing that, for every $i=\, \ldots, l$, we can enlarge the set $S$ to a set $S_i$ so that all the coefficients of $A_i$ are $S$-units. Moreover, the cardinality of the set $S_i$ is bounded by $\# S + 2(\deg A_i +1)^2 H_{\calC}(A_i)$. In particular, we can assume that we have extended the set $S$ so that every coefficient of a factor of $A$ is a $S$-unit, and the new set $S$ has cardinality bounded in terms of $\deg A$ and $H_{\calC}(A)$ (since the height of a factor $A_i$ is bounded in terms of the height of $A$; see e.g. \cite[Theorem 1.7.2]{BombieriGubler}). Therefore, from now on, we will assume that for every $i$, $A_i(X,Y) \in \calO_S^*[X,Y]$.\medskip

 We want to prove a bound for the number of multiple zeros of $A(u_1,u_2)$ in terms of $H_{\calC}(u_1)$ and $H_{\calC}(u_2)$. This is equivalent to bound the multiple zeros of every irreducible factor $A_i(u_1, u_2)$ together with a bound on the number of common zeros of $A_i(u_1, u_2)$ and $A_j(u_1, u_2)$ for every pair $1 \leq i < j \leq l$.
 
\subsection*{Number of multiple zeros of an irreducible polynomial} 

We begin by proving the bound for the number of multiple zeros of an irreducible polynomial, which is the content of the following proposition.

\begin{proposition}
  \label{prop:A_irred}
In the same setting as above, let $A(X,Y) \in \calO_S^*[X,Y]$ be an irreducible polynomial and let $\varepsilon > 0$ a positive real number. Then, there exist constants $\Ctwo = \Ctwo(\deg A, H_{\calC}(A),$ $\varepsilon)$ and $\Cirs~=~ \Cirs(\deg A, \varepsilon )$ such that, for every pair $(u_1,u_2) \in (\calO_S^*)^2$ with
   \[
     \max \{ H_{\calC}(u_1), H_{\calC}(u_2)\} \geq \Ctwo \max \{ 1, \chi_S(\calC) \},
   \]
at least one of the following holds:
   \begin{itemize}
	   \item the $S$-units $u_1,u_2$ verify a relation of the form $u_1^r u_2^s = \gamma$ for a pair of integers $(r,s) \in \Z^2 \setminus \{(0,0)\}$ such that $\max \{ \lvert r \rvert, \lvert s \rvert \} \leq \Cirs$ and $\gamma$ is an algebraic function of the coefficients of $A$ independent of $u_1, u_2$;
	   \item the rational function $A(u_1,u_2)$ verifies
   \begin{equation}
     \sum_{P \in \calC \setminus S} \max \{ 0, \ord_P(A(u_1,u_2)) -1 \} \le \varepsilon \max\{ H_{\calC}(u_1),H_{\calC}(u_2) \}.
	\label{eq:propA_irr}
\end{equation}
	\end{itemize}
\end{proposition}

\begin{proof}[Proof of Proposition \ref{prop:A_irred}]
\smallskip


 Note that, if $A(X,Y)$ does not depend on $X$ or on $Y$, a bound of the desired form is immediate. Therefore we will assume that $A$ depends nontrivally on both $X$ and $Y$.

Let us consider the polynomial $B(X,Y)$ given by:
\begin{equation} \label{eq:Bshape}
  B(X,Y) = \sum_{i+j \le \deg A} \lambda_{ij} X^i Y^j \left( i \frac{u_1'}{u_1} + j \frac{u_2'}{u_2} + \frac{\lambda_{ij}'}{\lambda_{ij}} \right);
\end{equation}
then, we have that $A'(u_1, u_2)=B(u_1, u_2)$.
We can enlarge $S$ to a set $S'$ including the set $T$ defined in Lemma \ref{lem:deriv} so that all the coefficients of $B$ have no poles outside $S'$, i.e. $B(X,Y) \in \calO_{S'}[X,Y]$ (notice the presence of the derivatives of the coefficients $\lambda_{ij}$ in the expression of $B(X,Y)$). Using Lemma \ref{lem:deriv}, we can bound the cardinality of $S'$ by
\begin{equation} \label{eq:S2}
\# S' \le \max \{1, \chi_{S}(\calC) \},
\end{equation}
and so
\begin{equation} \label{eq:chi2}
\max\{1,\chi_{S'}(\calC)\} \le 2 \max\{1, \chi_{S}(\calC)\}.
\end{equation}
 
\noindent We rewrite \eqref{eq:propA_irr} as 
\begin{align*}
\sum_{P \in \calC \setminus S} \max \{ 0, \ord_P(A(u_1,u_2)) -1 \} = & \sum_{P \in \calC \setminus S'} \max \{ 0, \ord_P(A(u_1,u_2)) -1 \} \\
 + & \sum_{P \in S' \setminus S} \max \{ 0, \ord_P(A(u_1,u_2)) -1 \}. \nonumber
\end{align*}
We will estimate the two sums separately, dividing the proof into several steps similarly to \cite{CZConic}. Moreover, we will explicitly compute all the constants involved at every step showing that they depend only on $\deg A$, $H_{\calC}(A)$ and $\varepsilon$ as wanted.\medskip

\paragraph{\textbf{Step 1.}} \emph{Either the two polynomials $A(X,Y)$ and $B(X,Y)$ are coprime, or $u_1$ and $u_2$ satisfy a relation of the form $u_1^r u_2^s = \gamma$, with $\gamma$ an algebraic function of the coefficients of $A$ and $\max\{|r|, |s|\}\le \deg A$.}

\begin{proof}
First, notice that by assumption $A$ is irreducible and neither $\partial A/\partial X$ nor $\partial A/\partial Y$ is identically zero, hence $A(X,Y)$ has at least two monomials. Suppose that the two polynomials $A$ and $B$ are not coprime. Since $\deg B(X,Y)\le \deg A(X,Y)$, there exists $a\in \kappa(\calC)^*$ such that $B(X,Y)=aA(X,Y)$. Using \eqref{eq:Bshape}, we have that $i \frac{u_1'}{u_1}+j \frac{u_2'}{u_2}+ \frac{\lambda_{ij}'}{\lambda_{ij}}=a$ for every $i,j$ with $\lambda_{ij}\neq 0$. Let us consider two monomials with $(i,j)\neq (h,k)$; we have
\begin{equation*}
  i \frac{u_1'}{u_1}+j \frac{u_2'}{u_2}+ \frac{\lambda_{ij}'}{\lambda_{ij}}=h \frac{u_1'}{u_1}+k \frac{u_2'}{u_2}+ \frac{\lambda_{hk}'}{\lambda_{hk}},
\end{equation*}
hence 
\begin{equation*}
  (i-h) \frac{u_1'}{u_1} +(j-k) \frac{u_2'}{u_2}=\frac{\lambda_{hk}'}{\lambda_{hk}}- \frac{\lambda_{ij}'}{\lambda_{ij}},
\end{equation*}
which gives a relation of multiplicative dependence between $u_1$ and $u_2$, i.e.
\[
u_1^{i-h} u_2^{j-k}= \mu \lambda_{hk} \lambda_{ij}^{-1},
\]
where $\mu \in \kappa^\times$. Taking $r=i-h$, $s=j-k$ and $\gamma =\mu \lambda_{hk} \lambda_{ij}^{-1}$, we have a relation $u_1^ru_2^s=\gamma$ with $\gamma$ an algebraic function of the coefficients of $A$ and $\max\{|r|,|s|\}\le \deg A$. This proves Step 1.\end{proof}

From now on, we will therefore assume that $A$ and $B$ are coprime.
\medskip

As $A(u_1,u_2)' = B(u_1,u_2)$, we have that, for every $P \not \in S'$,
\begin{equation} 
	\max \{ 0, \ord_P(A(u_1,u_2)) -1 \} \leq  \min \{ \ord_P(A(u_1,u_2)), \ord_P(B(u_1,u_2)) \},
  \label{eq:th12AB}
\end{equation}
since both $A(u_1, u_2)$ and $B(u_1, u_2)$ are $S'$-integers and so the term of the right hand-side is nonnegative.
\medskip

Let $F(X):=\mbox{Res}_Y(A(X,Y), B(X,Y)) \in \calO_{S'}[X]$ and  $G(Y):= \mbox{Res}_X(A(X,Y), B(X,Y))\in \calO_{S'}[Y]$ be the two resultants of $A$ and $B$ with respect to $Y$ and $X$. If $A$ and $B$ are coprime, then $F$ and $G$ do not vanish identically. Moreover, since $F(X)$ and $G(Y)$ are linear combinations of $A(X,Y)$ and $B(X,Y)$ over $\calO_{S'}[X,Y]$, we have that for every $ P \not \in S'$,
\begin{equation} \label{eq:AB}
\min \{ \ord_P(A(u_1,u_2)), \ord_P(B(u_1,u_2)) \} \le \min \{ \ord_P(F(u_1)), \ord_P(G(u_2)) \}.
\end{equation}
It is then enough to bound the gcd of $F(u_1)$ and $G(u_2)$. In order to do this, we first need to prove a bound for the degrees of $F$ and $G$ and for the heights of their coefficients. This is the content of Step 2. \medskip

\paragraph{\textbf{Step 2.}} \emph{There exist positive constants $\Cfour$ and $\Cfive$ such that the heights of $B(X,Y)$, $F(X)$ and $G(Y)$ are bounded by $\Cfour \max\{ 1, \chi_S(\calC) \}$ and the degrees of $B(X,Y)$, $F(X)$ and $G(Y)$ are bounded by $\Cfive$.}

\begin{proof}
Given the expression for $B(X,Y)$, we can bound the height of each of its coefficients as
\[
	H_{\calC}\left(\frac{u_1'}{u_1}\right) + H_{\calC}\left(\frac{u_2'}{u_2}\right) + H_{\calC}\left(\frac{\lambda_{ij}'}{\lambda_{ij}}\right) + H_{\calC}(\lambda_{ij}) \leq 3 \chi_{S'}(\calC) + H_{\calC}(A),
\]
since the first three terms are bounded by the Euler characteristic of $\calC \setminus S'$ using Lemma \ref{lem:deriv}.
Consequently, $H_{\calC}(B) \le 3 \max\{1, \chi_{S'}(\calC)\}  + H_{\calC}(A)$. Moreover, by definition, $F(X)$ and $G(Y)$ can be expressed as the determinant of a $N \times N$ matrix whose entries are the monomials appearing in $A(X,Y)$ and $B(X,Y)$, and with $N=\deg_Y A(X,Y) + \deg_Y B(X,Y)$ for $F$ and $N=\deg_X A(X,Y)+ \deg_X B(X,Y)$ for $G$ respectively. Using that $\deg_Y B(X,Y)\le \deg_Y A(X,Y)$ and $\deg_X B(X,Y)\le \deg_X A(X,Y)$, we have that the heights of $F$ and $G$ are bounded as follows:
\[
H_{\calC}(F) \leq 2 \deg_Y A \left( 3 \max \{ 1, \chi_{S'}(\calC) \} + H_{\calC}(A) \right)\quad \text{and}\quad  H_{\calC}(G) \leq 2 \deg_X A \left( 3 \max \{ 1, \chi_{S'}(\calC) \} + H_{\calC}(A) \right).
\]
Using \eqref{eq:chi2}, the first estimate is proved by taking
\[
\Cfour:= 2 \max \{\deg_X A(X,Y), \deg_Y A(X,Y) \} (6 + H_{\calC}(A)). 
\] 
Finally, since the degrees of $F(X)$ and $G(Y)$ are bounded by $2\deg_X A(X,Y) \deg_Y A(X,Y)$, and 
\[
	\deg B(X,Y) = \deg_X B(X,Y) + \deg_Y B(X,Y) \leq 2\max \{\deg_X A(X,Y), \deg_Y A(X,Y)\},
\]
we have that $\Cfive: = 2 \deg_X A(X,Y) \deg_Y A(X,Y)$ gives the desired estimate. \end{proof}

Next, we need to factor $F(X)$ and $G(Y)$; this can be done in a suitable finite extension of $\kappa(\calC)$. However, in order to bound the gcd of $F(u_1)$ and $G(u_2)$, we have to estimate the degree of this extension and the height of the roots of $F(X)$ and $G(Y)$. This can be done as follows. \medskip

	\paragraph{\textbf{Step 3.}} \emph{There exist a cover $p:\calE \to \calC$ of degree bounded by $2\Cfive$ and a finite set $U \subset \calE$ such that $F(X)$ and $G(Y)$ splits over $\kappa(\calE)$ into linear factors and their roots are $U$-units. Moreover, there exist two constants $\Csix, \Cseven>0$ such that }
\begin{equation} \label{eq:U}
  \chi(\calE) \leq \Csix \max \{ 1, \chi_S(\calC) \} \quad \mbox{ and } \quad \# U \le \Cseven \max\{ 1, \chi_S(\calC)\} .
\end{equation}
\begin{proof}
	If the polynomials $F(X)$ and $G(Y)$ split over $\kappa(\calC)$, then the conclusion holds trivially, so we will assume that this is not the case. 

	Let us define $\kappa(\calE)$ to be the splitting field of $F(Z)G(Z)$ and let us denote by $p: \calE \to \calC$ the cover corresponding to the field extension. Then, $\deg p \leq \deg F + \deg G \leq 2 \Cfive$. To estimate the Euler characteristic of $\calE$ we can use the Riemann-Hurwitz formula. First, note that the ramification of $p$ can arise only over zeros or poles of the discriminants of the irreducible factors of $F$ and $G$. By construction, the poles are contained in the set $S'$; on the other hand, the number of zeros is bounded by the heights of the discriminants. Recall that the discriminant of a degree $n$ polynomial $h$ is an homogeneous polynomial in the coefficients of $h$ of degree $2n-2$.  Therefore, its height is bounded by $2 H(h) \deg h$. In our case, the heights of $F(X)$ and $G(Y)$ are both bounded by $\Cfour \max\{1, \chi_S(\calC)\}$ as proved in Step 2, and therefore the total number of zeros is bounded by $2\Cfour\Cfive \chiSC$. This, together with \eqref{eq:S2}, implies that the cardinality of the support of the ramification divisor is bounded by $2 \Cfour\Cfive \chiSC + \# S' \leq (2\Cfour \Cfive +1)  \chiSC$. \\

Since each ramification index is at most $\Cfive$ (as $\deg p \le 2\Cfive$), the total ramification of $p$ is bounded by $\Cfive (2\Cfour \Cfive +1) \chiSC$. Applying the Riemann-Hurwitz formula, we get
\begin{equation*}
  \chi(\calE) \leq (\deg p) \chi(\calC) + \Cfive (2\Cfour \Cfive +1) \chiSC  \leq \Csix \chiSC, 
\end{equation*}
where $\Csix:=\Cfive(2\Cfour\Cfive +3)$.\\

Let us define the set $U'$ as the set of zeros of the constant and leading terms of both $F$ and $G$ and $U'' :=  U' \cup S'$. Note that the cardinality of $U'$ is bounded by $2 (H_{\calC}(F) + H_{\calC}(G)) \leq 4 \Cfour \max \{1, \chi_S(\calC)\}$, and therefore the cardinality of $U''$ is bounded by $(4 \Cfour + 1) \max \{ 1, \chi_S(\calC) \}$. We define $U:= p^{-1}(U'')$. Notice that, by construction, the roots of $F$ and $G$ are $U$-units in $\kappa(\calE)$; this follows from the fact that the coefficients of $F$ and $G$ are $S'$-integers and the product of all the roots of both $F$ and $G$ is a $U''$-unit. The cardinality of $U$ is bounded by $(\deg p) (\# U'')$, i.e.
\begin{equation*}
   \# U \le \Cseven \max\{ 1, \chi_S(\calC)\},
\end{equation*}
with $\Cseven:= 2\Cfive (4\Cfour + 1)$, which completes the proof.
\end{proof}

\medskip

Recall that, in the proof of the previous step, defined the set $U''$ as $U' \cup S'$, where $U'$ is the set of zeros of the constant and leading terms of both $F$ and $G$, and the set $U$ as $p^{-1}(U'')$, where $p: \calE \to \calC$ is the cover defined in Step 3. 


\noindent Using that $U''\supseteq S' \supseteq S$, we can split the sum in \eqref{eq:propA_irr} as
\begin{align*}
\sum_{P \in \calC \setminus S } \max \{ 0, \ord_P(A(u_1,u_2)) -1 \} \le & \sum_{P \in \calC \setminus U'' } \max \{ 0, \ord_P(A(u_1,u_2)) -1 \} \\
+ & \sum_{P \in U'' \setminus S } \max \{ 0, \ord_P(A(u_1,u_2)) -1 \}.
\end{align*}

By \eqref{eq:th12AB} and the fact that $p(U)=U''$, we can bound the first sum as 
\begin{align*}
\sum_{P \in \calC \setminus U'' } \max \{ 0, \ord_P(A(u_1,u_2)) -1 \} \le & \sum_{P \in \calC \setminus U'' } \min \{ \ord_P(A(u_1,u_2)), \ord_P(B(u_1,u_2)) \} \\
 = & \frac{1}{[\kappa(\calE): \kappa(\calC) ]} \sum_{P \in \calE \setminus U} \min \{ \ord_P(A(u_1,u_2)), \ord_P(B(u_1,u_2)) \},
\end{align*}
hence we get
\begin{align} \label{eq:C-S}
\sum_{P \in \calC \setminus S} \max \{ 0, \ord_P(A(u_1,u_2)) -1 \} & \le  \sum_{P \in \calE \setminus U} \min \{ \ord_P(A(u_1,u_2)), \ord_P(B(u_1,u_2)) \} \\
& + \sum_{P \in U'' \setminus S} \max \{ 0, \ord_P(A(u_1,u_2)) -1 \}. \nonumber
\end{align}
We now want to bound the first sum of the right hand side of \eqref{eq:C-S} using the polynomials $F$ and $G$.
\noindent Since $\kappa(\calE)$ is the splitting field of $F(Z)G(Z)$, we can rewrite the two polynomials in $\kappa(\calE)$ as
\begin{equation} \label{eq:factFG}
  F(X) = \eta (X - \alpha_1)\cdots (X-\alpha_m) \quad  \mbox{ and } \quad G(Y) = \nu (Y - \beta_1) \cdots (Y - \beta_n),
\end{equation}
where $\eta,\alpha_1,\dots,\alpha_m$ and $\nu,\beta_1,\dots,\beta_n$ are $U$-units in $\kappa(\calE)$ as proved in the previous step.

\noindent For every point $P \in \calE \setminus U$ we have that
\begin{equation*}
\min \{ \ord_P(F(u_1)), \ord_P(G(u_2)) \} \le \sum_{(i,j)} \min \{ \ord_P(u_1 - \alpha_i), \ord_P(u_2 - \beta_j) \},
\end{equation*}
using the general property that $\min \{ \sum_i a_i, \sum_j b_j\} \le \sum_{(i,j)} \min \{a_i, b_j\} $; hence, by \eqref{eq:AB}, for every $P\in \calE \setminus U$ we have 
\begin{equation} \label{eq:step4_1}
\min \{ \ord_P(A(u_1,u_2)), \ord_P(B(u_1,u_2)) \} \le \sum_{(i,j)} \min \{ \ord_P(u_1 - \alpha_i), \ord_P(u_2 - \beta_j) \}.
\end{equation}

Let us now define the set $\calZ = \{ \alpha,\beta \in \kappa(\calE): A(\alpha,\beta) = 0 = B(\alpha,\beta),\ \alpha\beta \neq 0 \}$. This set is in principle strictly smaller than the set of pairs $(\alpha_i,\beta_j)$ such that $F(\alpha_i) = G(\beta_j) = 0$ that we consider in the sum; however, we will show that, eventually replacing $U$ with a bigger set, the inequality \eqref{eq:step4_1} remains true if we restrict the sum to the the pairs $(\alpha_i, \beta_j) \in \calZ$. This is the content of the following step.\medskip

%
%
  \paragraph{\textbf{Step 4.}}  \textit{There exist a finite set $V \supseteq U$ and a constant $C_V>0$ such that $\#V \leq C_V \max \{1, \chi_S(\mathcal C)\}$ and, for every $P \in \calE \setminus V$,}
\begin{equation}
  \label{eq:setZ}
\min \{ \ord_P(A(u_1,u_2)), \ord_P(B(u_1,u_2)) \} \le \sum_{(\alpha,\beta) \in \calZ} \min \{ \ord_P(u_1 - \alpha_i), \ord_P(u_2 - \beta_j) \}.
\end{equation}
\begin{proof}
	Let $V$ be the subset of $\calE$ obtained by enlarging $U$ such that, for every $i$ and $j$, $A(\alpha_i, \beta_j)$ and $B(\alpha_i, \beta_j)$ are $V$-units, whenever they are not zero.  The cardinality of the set $V$ can be bounded as follows: since $\alpha_i$ and $\beta_j$ are $U$-units which are roots of $F$ and $G$ respectively, their heights are bounded by $\max\{ H_{\calE}(F), H_{\calE}(G) \}$, which, using Step 2 and Step 3, is bounded by $2 \Cfour \Cfive \max \{1, \chi_S(\mathcal C)\}$. Moreover, we can always bound the height of $A(\alpha_i, \beta_j)$ and $B(\alpha_i, \beta_j)$ in terms of the degree of $A$ and the maximum of the heights of their monomials; using that $H_{\calC}(A)\le \Cfour$, this implies that
\begin{equation*}
  H_{\calE}(A(\alpha_i, \beta_j)) \le H_{\calE}(A)+(\deg A) (H_{\calE}(\alpha_i)+H_{\calE}(\beta_j)) \le 4\Cfour \Cfive(1 + \deg A) \max \{1, \chi_S(\mathcal C)\}, 
\end{equation*}
and similarly, since $\deg B \leq \deg A$ and $H_{\calE}(B)\le \Cfour \max \{1, \chi_S(\mathcal C)\}$,
\begin{equation*}
H_{\calE}(B(\alpha_i, \beta_j)) \le H_{\calE}(B) +(\deg A) (H_{\calE}(\alpha_i)+H_{\calE}(\beta_j)) \le 4 \Cfour \Cfive(1 + \deg A)\max \{1, \chi_S(\mathcal C)\} . 
\end{equation*}
Taking into account that the number of pairs $(\alpha_i, \beta_j)$ is bounded by $(\deg F)(\deg G)$, which is bounded by $\Cfive^2$ by Step 2, the former two inequalities give
\begin{equation} \label{eq:cardV}
\# V \le \# U + \Cfive^2 (H_{\calE}(A(\alpha_i, \beta_j))+ H_{\calE}(B(\alpha_i, \beta_j))) \le C_V \max \{1, \chi_S(\mathcal C)\},
\end{equation}
with $C_V:= \Cseven + 8 \Cfour \Cfive^3( 1 +  \deg A)$.
\medskip

Let us fix $P \in \calE \setminus V$. To prove \eqref{eq:setZ} we can assume without loss of generality that $\min \{ \ord_P(A(u_1,u_2)) , $ $ \ord_P(B(u_1,u_2)) \}>0$, otherwise the inequality is trivial since every term of the sum in the right hand side of \eqref{eq:setZ} is nonnegative. In particular, we can assume that $\ord_P(A(u_1,u_2)) > 0$. We will show that, if $(\alpha_i, \beta_j) \notin \calZ$, then
\[
  \min \{ \ord_P(u_1 - \alpha_i), \ord_P(u_2 - \beta_j) \} = 0,
\]
hence proving the claim. To see this, consider $(\alpha_i,\beta_j) \notin \calZ$, so for example $A(\alpha_i,\beta_j) \neq 0$.  By definition of $V$, $A(\alpha_i, \beta_j)$ is a $V$-unit, which implies that $\ord_P(A(\alpha_i, \beta_j))=0$ since $P \notin V$. But this implies that $\min \{ \ord_P(u_1 - \alpha_i), \ord_P(u_2 - \beta_j) \}=0$, otherwise both the differences $u_1-\alpha_i$ and $u_2-\beta_j$ would have a zero in $P$, and therefore since $A(u_1, u_2)$ has a zero in $P$, also $A(\alpha_i, \beta_j)$ would have a zero in $P$, which contradicts the fact that $\ord_P(A(\alpha_i, \beta_j))=0$.
\end{proof}

\noindent Using the bound for $\chi(\calE)$ obtained in \eqref{eq:U} we get
\begin{equation}
  \chi_V(\calE) = \chi(\calE)+ \# V \leq (\Csix + C_V) \max\{ 1, \chi_S(\calC) \}.
  \label{eq:chi_vE}
\end{equation}

\noindent By \eqref{eq:setZ}, Step 2, and inverting the order of summation, we have:
\begin{align} \label{eq:Step5}
 \sum_{P \in \calE \setminus V} \min \{ \ord_P(A(u_1,u_2)), \ord_P(B(u_1,u_2)) \} & \le \sum_{P \in \calE \setminus V} \sum_{(\alpha, \beta) \in \calZ} \min \{ \ord_P(u_1 - \alpha), \ord_P(u_2 - \beta) \} \nonumber \\
& \le \Cfive^2 \max_{(\alpha, \beta) \in \calZ}  \sum_{P \in \calE \setminus V} \min \{ \ord_P(u_1 - \alpha), \ord_P(u_2 - \beta) \}. 
\end{align}
Let us define $(\alpha_{\hat{\imath}}, \beta_{\hat{\jmath}})$ as a pair in $\calZ$ for which the maximum is obtained.

In order to estimate the right hand side of \eqref{eq:Step5}, we will apply the gcd result of \cite{CZConic} to the $V$-units $u_1/\alpha_{\hat{\imath}}$ and $u_2/\beta_{\hat{\jmath}}$. We distinguish two cases according to whether the $V$-units are multiplicative independent modulo constants or not. \medskip

\paragraph{\textbf{Step 5.}} \emph{Assume that $u_1/\alpha_{\hat{\imath}}$ and $u_2/\beta_{\hat{\jmath}}$ are multiplicatively independent modulo constants; then, there exist a constant $\Cnine>0$ such that, if $\max \{ H_\calC(u_1), H_\calC(u_2) \} \geq \Cfour \max\{ 1, \chi_S(\calC) \}$, then}
\begin{equation*}
 \sum_{P \in \calE \setminus V} \min \{ \ord_P(A(u_1,u_2)), \ord_P(B(u_1,u_2)) \} \leq \Cnine \max\left\{ H_\calC (u_1), H_\calC (u_2) \right\}^{2/3} \max\{ 1, \chi_S(\calC) \}^{1/3}.
\end{equation*}

\begin{proof}
	Since we are assuming that the $V$-units $u_1/\alpha_{\hi}$ and $u_2/\beta_{\hj}$ are multiplicative independent modulo constants, we get that \cite[Corollary 2.3]{CZConic} implies that	
\begin{equation} \label{eq:step4.1}
	\sum_{P \in \calE \setminus V} \min \{ \ord_P(u_1 - \alpha_{\hat{\imath}}), \ord_P(u_2 - \beta_{\hat{\jmath}}) \} \leq 3 \sqrt[3]{2} \max \left\{ H_\calE \left(\frac{u_1}{\alpha_{\hat{\imath}}}\right), H_\calE \left(\frac{u_2}{\beta_{\hat{\jmath}}}\right) \right\}^{2/3} \chi_V(\calE)^{1/3}.
\end{equation}

\noindent Using the bound of $\chi_V(\calE)$ in \eqref{eq:chi_vE}, we can rewrite \eqref{eq:Step5} using \eqref{eq:step4.1} as
\begin{equation}\label{eq:st4}
\small{
\sum_{P \in \calE \setminus V} \min \{ \ord_P(A(u_1,u_2)), \ord_P(B(u_1,u_2)) \}  \leq  3\sqrt[3]{2}\Cfive^2 (\Csix + C_V)^{1/3}\max \left\{ H_\calE \left(\frac{u_1}{\alpha_{\hat{\imath}}}\right), H_\calE \left(\frac{u_2}{\beta_{\hat{\jmath}}}\right) \right\}^{2/3}\max\{1,\chi_S(\calC)\}^{1/3}.}
\end{equation}
To prove the statement we want to relate the heights of $u_1/\alpha_{\hi}$ and $u_2/\beta_{\hj}$ with the heights of $u_1$ and $u_2$. Since $\alpha_{\hi}$ and $\beta_{\hj}$ are roots of the polynomials $F$ and $G$, by Step 2 we obtain that
  \begin{equation*}
    \max \{ H_\calE(\alpha_{\hi}), H_\calE(\beta_{\hj}) \} \leq 2\Cfour \Cfive \max \{ 1, \chi_S(\calC) \}.
  \end{equation*}  
  Therefore, if 
  \begin{equation}
   \max\{ H_\calE(u_1), H_\calE(u_2) \} \geq 2\Cfour \Cfive \max \{ 1, \chi_S(\calC) \},
    \label{eq:step5.bound}
  \end{equation}
  we have that
  \begin{equation}
    H_\calE \left( \dfrac{u_1}{\alpha_{\hi}} \right) \leq H_\calE (u_1) + H_\calE \left(\alpha_{\hi} \right) \leq 2 \max\{ H_\calE(u_1), H_\calE(u_2) \}= 4 C_4 \max\{ H_\calC(u_1), H_\calC(u_2) \},
    \label{eq:step5.2}
  \end{equation}
and similarly for $u_2/\beta_{\hj}$. Notice that 
 \eqref{eq:step5.bound} is equivalent to 
 \begin{equation*}
 \max\{ H_\calC(u_1), H_\calC(u_2) \} \geq \Cfour  \max \{ 1, \chi_S(\calC) \};
 \end{equation*}
 under this assumption, using \eqref{eq:step5.2}, we can rewrite \eqref{eq:st4} as
 \begin{equation*}
	  \sum_{P \in \calE \setminus V} \min \{ \ord_P(A(u_1,u_2)), \ord_P(B(u_1,u_2)) \}
    \leq \Cnine \max\left\{ H_\calC(u_1), H_\calC (u_2) \right\}^{2/3} \max\{ 1, \chi_S(\calC) \}^{1/3}, \nonumber 
  \end{equation*}
  where we set $\Cnine := 3 \cdot 2^{5/3} \, \Cfive^{8/3} (\Csix + C_V)^{1/3}$.
\end{proof}

In the next step, we deal with the case in which the $V$-units $u_1/\alpha_{\hi}$ and $u_2/\beta_{\hj}$ are multiplicatively dependent modulo constants.  We observe that, as by construction $(\alpha_{\hat{\imath}}, \beta_{\hat{\jmath}}) \in \calZ$, we can apply Lemma \ref{lem:314} which implies that in this case either $u_1$ and $u_2$ satisfy a relation of the form $u_1^r u_2^s = \gamma$, where $\gamma$ is an algebraic function of the coefficients of $A$, or at least one between $u_1/\alpha_{\hat{\imath}}$ and $u_2/\beta_{\hat{\jmath}}$ is constant. However, using the bounds for $H_{\calC}(F)$ and $H_{\calC}(G)$ given in Step 2, we have that, if $\max\{ H_{\calC}(u_1), H_{\calC}(u_2) \} \ge \Cfour \max \{ 1, \chi_S(\calC) \}$, we can always assume to be in the case in which none of the quotients $u_1/\alpha_{\hat{\imath}}$ and $u_2/\beta_{\hat{\jmath}}$ is constant.

\medskip

\paragraph{\textbf{Step 6.}} \emph{Assume that $u_1/\alpha_{\hat{\imath}}, u_2/\beta_{\hat{\jmath}}$ are multiplicatively dependent and let $(u_1/\alpha_{\hat{\imath}})^r (u_2/\beta_{\hat{\jmath}})^s=\mu$ be a generating relation, for some $\mu \in \kappa^{\times}$. Then, if $\max\{ H_\calC(u_1), H_\calC(u_2) \} \geq \Cfour  \max \{ 1, \chi_S(\calC) \}$,
we have that $u_1^r u_2^s = \gamma$, where $\gamma$ is an algebraic function of the coefficients of $A$, and }
\begin{equation*} 
 \sum_{P \in \calE \setminus V} \min \{ \ord_P(A(u_1,u_2)), \ord_P(B(u_1,u_2)) \} \leq 4 \Cfive^3 \, \dfrac{\max \{ H_\calC(u_1), H_\calC (u_2) \}}{\max\{|r|,|s|\}}.
\end{equation*}
\begin{proof}
	Under the dependence assumption, \cite[Corollary 2.3]{CZConic} applied to the $V$-units $u_1/\alpha_{\hat{\imath}}$ and $u_2/\beta_{\hat{\jmath}}$ gives
\begin{equation} \label{eq:step6}
\sum_{P \in \calE \setminus V} \min \{ \ord_P(u_1 - \alpha_{\hat{\imath}}), \ord_P(u_2 - \beta_{\hat{\jmath}}) \} \leq \dfrac{\max \left\{ H_\calE \left(\frac{u_1}{\alpha_{\hi}}\right), H_\calE \left(\frac{u_2}{\beta_{\hj}}\right) \right\}}{\max \{ \lvert r \rvert, \lvert s \rvert \} }.
\end{equation}

As in the previous step, assuming that $\max \{ H_\calC(u_1), H_{\calC}(u_2) \} \geq \Cfour \max \{1, \chi_S(\calC) \}$, we can rewrite \eqref{eq:Step5}, using \eqref{eq:step6}, in terms of the heights of $u_1$ and $u_2$ as
\begin{equation*}
	\sum_{P \in \calE \setminus V} \min \{ \ord_P(A(u_1,u_2)), \ord_P(B(u_1,u_2)) \} \leq 4 \Cfive^3\, \dfrac{ \max \{ H_{\calC}(u_1), H_{\calC}(u_2) \}}{\max \{ \lvert r \rvert, \lvert s \rvert \}},
\end{equation*}
as wanted. Moreover, the relation $u_1^r u_2^s=\gamma$, with $\gamma$ an algebraic function of the coefficients of $A$ follows applying Lemma \ref{lem:314} and using the assumption that $\max \{ H_\calC(u_1), H_{\calC}(u_2) \} \geq \Cfour \max \{1, \chi_S(\calC) \}$.
  \end{proof}

Recall that our goal is to prove a bound on the number of multiple zeros of $A(u_1,u_2)$. This problem has been reduced to estimate the following sum:
\begin{align}
  \label{eq:splitv}
  \sum_{P \in \calC \setminus S} \max\{ 0, \ord_P(A(u_1,u_2)&-1) \}  \leq \sum_{P \in \calE \setminus V} \min\{ \ord_P(A(u_1,u_2)), \ord_P(B(u_1,u_2)) \} \nonumber \\ 
  & + \sum_{P \in V \setminus U} \max\{ 0, \ord_P(A(u_1,u_2))-1 \}    
  + \sum_{P \in U''\setminus S} \max\{ 0, \ord_P(A(u_1,u_2))-1 \} \nonumber \\
  & \leq \sum_{P \in \calE \setminus V} \min\{ \ord_P(A(u_1,u_2)), \ord_P(B(u_1,u_2)) \} \\
  & + \left[\kappa(\calE): \kappa(\calC)\right] \sum_{P \in p(V)} \max\{ 0, \ord_P(A(u_1,u_2)) \}, \nonumber 
\end{align}
since by construction $U=p^{-1}(U'')$ where $p: \calE \rightarrow \calC$ is the cover defined in Step 3.

In Steps 5 and 6 we proved an estimate for the first sum appearing in the right hand side. We are left with providing an estimate for the remaining sum, which we deal with in the following step.\medskip

\paragraph{\textbf{Step 7.}} \emph{There exists a constant $\Cten >0$ such that:}
\begin{equation*}
  \sum_{P \in p(V)} \max\{ 0, \ord_P(A(u_1,u_2)) \} \leq \Cten \max \{ 1, \chi_S(\calC) \}.
\end{equation*}
\begin{proof}
  Let us write the polynomial $A(u_1,u_2) = \theta_1 + \cdots +\theta_M$ as a sum of monomials, where no subsum vanishes. All the $\theta_i$ are $S$-units, in particular they are $p(V)$-units as $p(V) \supseteq S$; hence,
  \begin{equation}
    H_\calC(\theta_1: \cdots : \theta_M) \geq H_\calC (A(u_1,u_2)) = \sum_{P \in \calC} \max\{0, \ord_P(A(u_1,u_2)) \}.
    \label{eq:step6.1}
  \end{equation}
  On the other hand \cite[Lemma 3.11]{CZConic}, which follows from \cite[Theorem 1]{Zannier}, implies that
  \begin{equation}
    \sum_{P \in \calC \setminus p(V)}\max\{0, \ord_P(A(u_1,u_2)) \} \geq H_\calC (\theta_1: \cdots: \theta_M) - {M \choose 2} \chi_{p(V)}(\calC);
    \label{eq:step6.2}
  \end{equation}
  so, combining \eqref{eq:step6.1} and \eqref{eq:step6.2},
 we have
  \begin{equation}
    \sum_{P \in p(V)} \max\{0, \ord_P(A(u_1,u_2)) \} \leq {M \choose 2} \chi_{p(V)} (\calC).
    \label{eq:step6.3}
  \end{equation}
  The number of monomials of $A(u_1,u_2)$ is bounded by $(\deg A + 1)^2$, and $\chi_{p(V)}(\calC) \leq 2g(\calC) -2 + \# V$. Hence, using the estimate for $\#V$ obtained in \eqref{eq:cardV}, we have
  \begin{equation*}
    \chi_{p(V)}(\calC) \leq (1 + C_V) \max \{ 1, \chi_S(\calC) \}.
  \end{equation*}
  Therefore, \eqref{eq:step6.3} can be rewritten as
  \begin{equation*}
    \sum_{P \in p(V)} \max\{0, \ord_P(A(u_1,u_2)) \} \leq {(\deg A + 1)^2 \choose 2} (1 + C_V) \max \{ 1, \chi_S(\calC) \} .
  \end{equation*}
Taking $\Cten := {(\deg A + 1)^2 \choose 2} (1 + C_V)$, 
we have the desired estimate.
\end{proof}
\medskip
\paragraph{\textit{End of the proof of Proposition \ref{prop:A_irred}.}}
Fix $\varepsilon > 0$. From \eqref{eq:splitv}, we have that, either $u_1^r u_2^s = \gamma$ with $\gamma$ an algebraic function of the coefficients of $A$ and $\max\{ |r|, |s| \} \leq \deg A$ or
\begin{align*}
 \sum_{P \in \calC \setminus S} \max\{ 0, \ord_P(A(u_1,u_2)-1) \} & \le \sum_{P \in \calE \setminus V} \min\{ \ord_P(A(u_1,u_2)), \ord_P(B(u_1,u_2)) \}  \\
& + \left[\kappa(\calE): \kappa(\calC)\right] \sum_{P \in p(V)} \max\{ 0, \ord_P(A(u_1,u_2)) \}.
\end{align*}
In the case in which $u_1/\alpha_{\hi}$ and $u_2/\beta_{\hj}$ are independent modulo constants, by Step 5, if $\max \{ H_\calC(u_1), H_\calC(u_2) \} \geq \Cfour \max\{ 1, \chi_S(\calC) \}$, we get
\begin{equation*}
	\sum_{P \in \calC \setminus S} \max\{ 0, \ord_P(A(u_1,u_2)) - 1 \} \le \Cnine \max \{ H_\calC(u_1), H_\calC(u_2) \}^{2/3} \max \{ 1, \chi_S(\calC) \}^{1/3} + 2 \Cfive \Cten \max\{1, \chi_S(\calC) \}.
\end{equation*}
Therefore in this case we obtain that, if $ \max \{ H_{\calC}(u_1), H_{\calC}(u_2) \} \geq \Cind \max \{ 1, \chi_S(\calC) \}$ with 
\begin{equation*}
	\Cind:= \max \left \{\Cfour, \left( \frac{2\Cnine}{\varepsilon}\right)^3, \frac{4\Cfive \Cten}{\varepsilon} \right \}, 
\end{equation*}
then
\begin{equation*}
  \sum_{P \in \calC \setminus S } \max \left \{ 0, \ord_P(A(u_1,u_2))-1 \right \} \leq  \varepsilon \max \{ H_\calC(u_1) , H_\calC(u_2) \},
\end{equation*}
as wanted.

\medskip

On the other hand, in the case in which $u_1/\alpha_{\hi}$ and $u_2/\beta_{\hj}$ are dependent, by Step 6 we obtain that, if $\max \{ H_\calC(u_1), H_\calC(u_2) \} \geq \Cfour \max\{ 1, \chi_S(\calC) \}$, then
\begin{equation*}
\sum_{P \in \calC \setminus S} \max\{ 0, \ord_P(A(u_1,u_2)) - 1 \} \le 4\Cfive^3 \dfrac{\max \{ H_\calC(u_1), H_\calC(u_2) \}}{ \max \{ |r|, |s| \} } + 2 \Cfive \Cten \max\{1, \chi_S(\calC) \}.
\end{equation*}
This implies that, either $u_1^ru_2^s=\gamma$ with $\gamma$ an algebraic function of the coefficients of $A$ and $\max \{ |r|, |s| \} \leq 8 \Cfive^3/\varepsilon$, or, if $ \max \{ H_{\calC}(u_1), H_{\calC}(u_2) \} \geq \Cdep \max \{ 1, \chi_S(\calC) \}$ with 
\begin{equation*}
	\Cdep := \max \left \{\Cfour, \frac{4\Cfive \Cten}{\varepsilon} \right \}, 
\end{equation*}
then
\begin{equation*}
	\sum_{P \in \calC \setminus S } \max \left \{ 0, \ord_P(A(u_1,u_2))-1 \right \} \leq  \varepsilon \max \{ H_\calC(u_1) , H_\calC(u_2) \}.
\end{equation*}
The two cases imply that, either $u_1^r u_2^s = \gamma$ with $\max \{ |r|, |s| \} \leq \Cirs$ where
\[
	\Cirs := \max \{ \deg A, 8 \Cfive^3/\varepsilon \},
\]
or, if $\max \{ H_\calC(u_1), H_{\calC}(u_2) \} \geq \Ctwo \max \{ 1, \chi_S(\calC) \}$, with $\Ctwo := \max \{ \Cind, \Cdep \}$, we get
\begin{equation*}
	\sum_{P \in \calC \setminus S } \max \left \{ 0, \ord_P(A(u_1,u_2))-1 \right \} \leq  \varepsilon \max \{ H_\calC(u_1) , H_\calC(u_2) \},
\end{equation*}
finishing the proof.
\end{proof}

%


\subsection*{Number of common zeros of two irreducible polynomials}
 To conclude the proof of Theorem \ref{th:1.2} we need to prove a bound for the number of common zeros of two irreducible polynomials, which is the content of the following proposition.

\begin{proposition}
  \label{prop:A_red}
  In the same setting of Theorem \ref{th:1.2}, let $A_1,A_2 \in \calO_S^*[X,Y]$ be two coprime polynomials and let $\varepsilon > 0$ be a positive real number. Then, there exist constants $\Done = \Done(\deg A_1 , \deg A_2,$ $H_{\calC}(A_1), H_{\calC}(A_2), \varepsilon)$ and $\Dirs = \Dirs(\deg A_1, \deg A_2, \varepsilon )$ such that, for every pair $(u_1,u_2) \in (\calO_S^*)^2$ with
   \[
     \max \{ H_{\calC}(u_1), H_{\calC}(u_2)\} \geq \Done \max \{ 1, \chi_S(\calC) \},
   \]
at least one of the following holds:
   \begin{itemize}
	   \item the $S$-units $u_1,u_2$ verify a relation of the form $u_1^r u_2^s = \gamma$ for a pair of integers $(r,s) \in \Z^2 \setminus \{(0,0)\}$ such that $\max \{ \lvert r \rvert, \lvert s \rvert \} \leq \Dirs$ and $\gamma$ is an algebraic function of the coefficients of $A_1$ and $A_2$, independent of $u_1, u_2$;
	   \item the rational functions $A_1(u_1,u_2)$ and $A_2(u_1,u_2)$ verify
   \begin{equation}
     \sum_{P \in \calC \setminus S} \min \{ \ord_P(A_1(u_1,u_2)), \ord_P(A_2(u_1,u_2))  \} \le \varepsilon \max\{ H_\calC(u_1),H_\calC(u_2) \}.
     \label{eq:prop_red_2}
   \end{equation}
	\end{itemize}
%
\end{proposition}

\begin{proof}[Proof of Proposition \ref{prop:A_red}]

 As before, we note that, if one of the two polynomials $A_1,A_2$ does not depend on $X$ or on $Y$, a bound of the desired form is immediate. Therefore we will assume that both $A_1,A_2$ depend nontrivally on both $X$ and $Y$.
Let us write the polynomials $A_1(X,Y)$ and $A_2(X,Y)$ as 
\[
  A_1(X,Y) = \sum_{i+j \le \deg A_1} \lambda_{ij} X^i Y^j \quad \mbox{and} \quad A_2(X,Y) = \sum_{k+l \le \deg A_2} \mu_{kl} X^k Y^l, 
\]
where $\deg A_i(X,Y) = \deg_X A_i(X,Y) + \deg_Y A_i(X,Y)$ for $i=1,2$ as before. \\

The proof follows the same steps as in the proof of Proposition \ref{prop:A_irred}, where the polynomial $A_2$ plays the role of the polynomial $B$. Therefore we will only indicate the required adjustments and compute the corresponding constants. We remark that, as before, all the constants will depend only on the degree and the heights of the polynomials $A_1,A_2$ and $\varepsilon$.\medskip

Note that Step 1 is automatically verified since the two polynomials $A_1$ and $A_2$ are coprime. 

Let $F(X)=\mbox{Res}_Y(A_1(X,Y), A_2(X,Y)) \in \calO_{S}[X]$ and $G(Y)= \mbox{Res}_X(A_1(X,Y), A_2(X,Y))\in \calO_{S}[Y]$ be the two resultants of $A_1$ and $A_2$ with respect to $Y$ and $X$, which do not vanish identically in view of the coprimality of $A_1$ and $A_2$. As in the previous case, we have that, for every place $ P \not \in S$,
\begin{equation}\label{eq:A1A2_2}
\min \{ \ord_P(A_1(u_1,u_2)), \ord_P(A_2(u_1,u_2)) \} \le \min \{ \ord_P(F(u_1)),\ord_P(G(u_2)) \}.
\end{equation}
It is then enough to bound the gcd of $F(u_1)$ and $G(u_2)$. By definition of $F$ and $G$ we get the following bounds for their degrees and their heights.\medskip

\paragraph{\textbf{Step 2.}} \emph{There exist positive constants $\Dtwo$ and $\Dthree$ such that the heights of $A_1(X,Y), A_2(X,Y),$ $F(X)$ and $G(Y)$ are bounded by $\Dtwo \max\{1, \chi_S(\calC)\}$ and the degrees of $F(X)$ and $G(Y)$ are bounded by $\Dthree$.}
\begin{proof}
	The bounds are obtained similarly to Step 2 of Proposition \ref{prop:A_irred}. In this setting we get 
\begin{equation*}
	  \Dtwo := 2 (\deg A_1 + \deg A_2)(H_{\calC}(A_1) + H_{\calC}(A_2)) \quad \mbox{and} \quad \Dthree : = (\deg A_1 + \deg A_2)^2.
\end{equation*}
\end{proof}

In order to bound the gcd of $F(u_1)$ and $G(u_2)$, we have to estimate the degree of the splitting field of $F(Z)G(Z)$ and the height of the roots of $F$ and $G$. \medskip

\paragraph{\textbf{Step 3.}} \emph{There exist a cover $p: \calE \to \calC$ of degree bounded by $2\Dthree$ and a finite set $U \subset \calE$ such that $F(X)$ and $G(Y)$ splits over $\kappa(\calE)$ into linear factors and their roots are $U$-units. Moreover, there exist constants $\Dfour, \Dfive>0$ such that }
\begin{equation*}
  \chi(\calE) \leq \Dfour \max \{ 1, \chi_S(\calC) \} \quad \mbox{ and } \quad \# U \le \Dfive \max\{ 1, \chi_S(\calC)\} .
\end{equation*}
\begin{proof}
	Letting $\kappa(\calE)$ be the splitting field of $F(Z)G(Z)$ and using the same argument as in the previous case, one concludes taking $\Dfour := \Dthree(2\Dtwo \Dthree + 3)$ and $\Dfive := 2 \Dthree (4\Dtwo + 1)$.
\end{proof}

\noindent In $\kappa(\calE)$ we can rewrite the two polynomials as
\begin{equation*} 
  F(X) = \eta (X - \alpha_1)\cdots (X-\alpha_m) \quad  \mbox{ and } \quad G(Y) = \nu (Y - \beta_1) \cdots (Y - \beta_n),
\end{equation*}
where $\eta,\alpha_1,\dots,\alpha_m$ and $\nu,\beta_1,\dots,\beta_n$ are $U$-units in $\kappa(\calE)$ as proved in the previous step. This implies that, for every $P \in \calE \setminus U$,
\[
	\min \{ \ord_P(F(u_1)), \ord_P(G(u_2)) \} \leq \sum_{(i,j)} \min \{ \ord_P(u_1 - \alpha_i), \ord_P(u_2 - \beta_j) \},
\]
hence applying \eqref{eq:A1A2_2} it follows that
\[
	\min \{ \ord_P(A(u_1, u_2)), \ord_P(B(u_1,u_2)) \} \leq \sum_{(i,j)} \min \{ \ord_P(u_1 - \alpha_i), \ord_P(u_2 - \beta_j) \},
\]
As in the previous case, we want to restrict the sum on the right to the set $\calZ = \{ (\alpha,\beta) \in \calE : A_1(\alpha,\beta) = 0 = A_2(\alpha,\beta),\ \alpha\beta \neq 0 \} $. In the next step we prove that the same inequality holds when we restrict to points in a finite set $V \supseteq U$ of bounded cardinality. 
\medskip

\paragraph{\textbf{Step 4.}} \emph{There exist a finite set $V \supseteq U$ and a positive constant $D_V$ such that $\#V \leq D_V \max \{ 1, \chi_S(\calC) \}$ and, for every $P \in \calE \setminus V$},
\begin{equation} 
  \label{eq:st4_bis}
 \min \{ \ord_P(A_1(u_1,u_2)), \ord_P(A_2(u_1,u_2)) \} \le  \sum_{(\alpha,\beta) \in \calZ} \min \{ \ord_P(u_1 - \alpha), \ord_P(u_2 - \beta) \} .
\end{equation}
\begin{proof}
  The proof is identical to the Step 4 of Proposition \ref{prop:A_irred} and we can take
\begin{equation*}
  D_V := \Dfive + 4 \Dtwo \Dthree^3 \left( 2 + \deg A_1 + \deg A_2 \right).
\end{equation*}
\end{proof}


\noindent We can then split the sum \eqref{eq:prop_red_2} as
\begin{align}
\label{eq:st4_2}
 \sum_{P \in \calC \setminus S} \min \{ \ord_P(A_1(u_1,u_2)), \ord_P(A_2(u_1,u_2)) \} & \le  \sum_{P \in \calE \setminus V} \min \{ \ord_P(A_1(u_1,u_2)), \ord_P(A_2(u_1,u_2)) \}  \nonumber \\
& + 2\Dthree \sum_{P \in p(V) \setminus S} \min \{ \ord_P(A_1(u_1,u_2)), \ord_P(A_2(u_1,u_2)) \}.
\end{align}

\noindent Using \eqref{eq:st4_bis} and Step 2, we can rewrite the first sum as 
\begin{align} \label{eq:ihjh}
\sum_{P \in \calE \setminus V} \min \{ \ord_P(A_1(u_1,u_2)), \ord_P(A_2(u_1,u_2)) \} & \le \sum_{P \in \calE \setminus V} \sum_{(\alpha,\beta) \in \calZ} \min \{ \ord_P(u_1 - \alpha), \ord_P(u_2 - \beta) \} \nonumber \\ & \le \Dthree^2 \max_{(\alpha,\beta) \in \calZ} \sum_{P \in \calE \setminus V} \min \{ \ord_P(u_1 - \alpha), \ord_P(u_2 - \beta) \}.
\end{align}

  Let $(\alpha_{\hi}, \beta_{\hj}) \in \calZ$ be a pair for which the maximum is obtained. 
  We will estimate the right hand side of \eqref{eq:ihjh} using the gcd estimate of \cite[Corollary 2.3]{CZConic} and relating it to the heights of $u_1$ and $u_2$. We will consider the case when the two $V$-units $u_1/\alpha_{\hi}$ and $u_2/\beta_{\hj}$ are multiplicatively independent or not modulo constants separately. This is the content of the next two steps.\medskip

\paragraph{\textbf{Step 5.}}
\emph{Assume that $u_1/\alpha_{\hi}$ and $u_2/\beta_{\hj}$ are multiplicatively independent modulo constants. Then, there exists a constant $\Dsix>0$ such that, if $\max\{ H_{\calC}(u_1),H_{\calC}(u_2) \} \ge \Dtwo \max \{ 1, \chi_S(\calC) \}$, then}
\begin{equation}
  \label{eq:st5_2}
  \sum_{P \in \calE \setminus V} \min \{ \ord_P(A_1(u_1,u_2)), \ord_P(A_2(u_1,u_2)) \} \leq \Dsix \max\left\{ H_\calC (u_1), H_\calC (u_2) \right\}^{2/3} \max\{ 1, \chi_S(\calC) \}^{1/3}.
\end{equation}

\begin{proof}
	We apply \cite[Corollary 2.3]{CZConic} to \eqref{eq:ihjh}, noticing that, if $\max\{ H_{\calC}(u_1),H_{\calC}(u_2) \} \ge \Dtwo \max \{ 1, \chi_S(\calC) \}$, we can bound the heights of $u_1/\alpha_{\hi}$ and $u_2/\beta_{\hj}$ in terms of the heights of $u_1$ and $u_2$. Following the same computation as in the irreducible case, we get that \eqref{eq:st5_2} holds with
\begin{equation*}
  \Dsix := 3 \cdot 2^{5/3} \Dthree^{8/3} \left(\Dfour + D_V\right)^{1/3}. 
\end{equation*}
\end{proof}

  We are left with the case in which $u_1/\alpha_{\hi}$ and $u_2/\beta_{\hj}$ are multiplicatively dependent modulo constants, which is the content of the following step. 
  
\medskip
\paragraph{\textbf{Step 6.}}
\emph{Assume that $u_1/\alpha_{\hi}, u_2/\beta_{\hj}$ are multiplicatively dependent and let $(u_1/\alpha_{\hi})^r(u_2/\beta_{\hj})^s = \mu$ be a generating relation for some $\mu \in \kappa^\times$. Then, $u_1^r u_2^s = \gamma$, where $\gamma$ is an algebraic function of the coefficients of $A_1$ and $A_2$ and, if $\max\{ H_{\calC}(u_1),H_{\calC}(u_2) \} $  $\ge \Dtwo \max \{ 1, \chi_S(\calC) \}$, we have }
\begin{equation}
  \label{eq:st6}
  \sum_{P \in \calE \setminus V} \min \{ \ord_P(A_1(u_1,u_2)), \ord_P(A_2(u_1,u_2)) \} \leq 4 D_3^3\, \dfrac{\max\left\{ H_\calC (u_1), H_\calC (u_2) \right\}^{2/3}}{\max \{ |r|, |s| \}}.
\end{equation}

\begin{proof}
	As before, we apply \cite[Corollary 2.3]{CZConic} to \eqref{eq:ihjh}, and, if $\max\{ H_{\calC}(u_1),H_{\calC}(u_2) \} \ge \Dtwo \max \{ 1, \chi_S(\calC) \}$, we can bound the heights of $u_1/\alpha_{\hi}$ and $u_2/\beta_{\hj}$ in terms of the heights of $u_1$ and $u_2$. The same computation as in the irreducible case give \eqref{eq:st6}. Notice that since $(\alpha_{\hi},\beta_{\hj})$ is a zero of $A_1$ and $A_2$, this implies directly that $ u_1^r u_2^s= \gamma$ where $\gamma:=\mu\alpha_{\hi}^r\beta_{\hj}^s$ is an algebraic function of the coefficients of $A_1$ and $A_2$.
\end{proof}

In the previous steps we estimated the first sum in the right hand side of \eqref{eq:st4_2}; in the last step, we bound the second sum.\medskip

\paragraph{\textbf{Step 7.}} 
\emph{There exists a constant $\Dseven>0$ such that:}
\begin{equation*}
  \sum_{P \in p(V)} \min\{ \ord_P(A_1(u_1,u_2)), \ord_P(A_2(u_1,u_2)) \}\leq \Dseven \max \{ 1, \chi_S(\calC) \}.
\end{equation*}
\begin{proof}
	The main idea is that the number of common zeros of $A_1(u_1,u_2)$ and $A_2(u_1,u_2)$, counted with multiplicities, is bounded by the number of zeros of $A_1(u_1,u_2)$ counted with multiplicities, and therefore we can reduce the computation to the case of a single polynomial as in the irreducible case. Using the same computation of Step 7 of Proposition \ref{prop:A_irred}, we can take
  \begin{equation*}
    \Dseven := {(\deg A_1 + 1)^2 \choose 2} (1 + D_V).
  \end{equation*}
\end{proof}

\paragraph{\textit{End of the proof of Proposition \ref{prop:A_red}}.}
The end of the proof is analogous to the irreducible case: given $\varepsilon>0$ one considers \eqref{eq:st4_2} and uses Step 5 and Step 6 to bound the first sum, and Step 7 to bound the second sum. Then, one gets the desired conclusion with the constants
\[
	\Done:= \max \left\{ \Dtwo, \left(\frac{2\Dsix}{\varepsilon}\right)^{3}, \frac{4\Dthree\Dseven}{\varepsilon} \right\} \qquad \text{and} \qquad \Dirs := 8\Dthree^3/\varepsilon.
\]
 \end{proof}

\paragraph{\textit{End of the proof of Theorem \ref{th:1.2}}.} 
 
Given the factorization of the polynomial $A$ given in  \eqref{eq:factorization} we have that, either $u_1$ and $u_2$ satisfy a multiplicative dependence relation of the form $u_1^r u_2^s = \gamma$ with $\max\{ |r|, |s| \} \leq \Crs$ and $\Crs := \max \{ \Cirs, \Dirs \}$ or
\begin{align*}
	\sum_{P \in \calC \setminus S} \max \{ 0,  \ord_P(A(u_1,u_2)) - 1 \} & = \sum_{i=1}^l \sum_{P \in \calC \setminus S}\max \{0, \ord_P(A_i(u_1,u_2)) - 1 \} \\
	& + \sum_{1\leq i < j \leq l} \sum_{P \in \calC \setminus S}\min \{ \ord_P (A_i(u_1,u_2)), \ord_P(A_j(u_1,u_2)) \} \\
	&\leq (\deg A + 1)^2 \left( \max_i \sum_{P \in \calC \setminus S} \max \{0, \ord_P(A_i(u_1,u_2)) - 1 \}\right. \\
	& + \left. \max_{(i , j)} \sum_{P \in \calC \setminus S}\min \{ \ord_P (A_i(u_1,u_2)), \ord_P(A_j(u_1,u_2)) \} \right).
	\label{eq:fin}
\end{align*}

\noindent We obtain the bound \eqref{eq:th12} with
\[
	\Cone := \max \left\{\, \max_i\, C_1\left(A_i, \frac{\varepsilon}{2(\deg A + 1)^2}\right),\, \max_{(i,j)}\, D_1\left(A_i,A_j,\frac{\varepsilon}{2(\deg A + 1)^2}\right) \right\},
\]
where $C_1(A_i, \delta)$ is the constant appearing in Proposition \ref{prop:A_irred} applied to the polynomial $A_i$ and the real number $\delta$ and $D_1(A_i,A_j,\delta)$ is the constant appearing in Proposition \ref{prop:A_red} applied to the polynomials $A_i$ and $A_j$ and the real number $\delta$.
\end{proof}

\begin{remark}
  Recently in \cite{Levin2018}, Levin obtained a generalization of gcd results over number fields as \cite{BCZ03, CZ05} for polynomials with an arbitrary number of variables. These results have been used to prove some cases of the Lang-Vojta conjectures in higher dimension both over number fields \cite{Levin2018} and function fields \cite{CLZ19}. 
\end{remark}

\section{The Ramification Divisor}\label{sec:ram}
%
%
In this section we study the contribution of the ramification divisor of the map $\pi\restriction_{\calX \setminus \calD}$ to the height of a section $\sigma:\calC \to \calX$ with $\supp( \sigma^*\calD) \subseteq S$, where $\calX$ is a threefold as in Section \ref{sec:3folds} and $S \subset \calC$ is a finite set of points. Recall that we denoted by $Z$ the closure of the ramification divisor of the finite map $\pi\restriction_{\calX \setminus \calD}: \calX \setminus \calD \to \G_m^2 \times \calC$. The image $\pi(Z \setminus \calD)$ will be defined by the vanishing of a certain polynomial $A \in \kappa(\calC)[X,Y]$. We will use Theorem \ref{th:1.2} to derive a bound for the degree of the pull-back of the ramification divisor. This is the content of the following proposition.

\begin{proposition}\label{prop:ram}
In the setting above, for every $\varepsilon > 0$ there exists a constant $C = C(\deg A, H_{\calC}(A), \deg \pi, \varepsilon)$ such that, for every section $\sigma: \calC \to \calX$ of height $H \geq C \, \max\{ 1, \chi_S(\calC) \}$ such that $\supp (\sigma^*\calD) \subseteq S$ and $\sigma(\calC)$ is not contained in $Z$, the degree of $\sigma^*(Z) \setminus S$ satisfies
 \begin{equation*}
   \deg (\sigma^*(Z) \setminus S) \leq \varepsilon H.
 \end{equation*}
\end{proposition} 
   
 \begin{proof}
    The statement is trivial if $\sigma(\calC) \cap Z$ is empty, and without loss of generality we can assume $Z$ irreducible by applying the same argument to each irreducible component.
    By construction $\pi(Z\setminus \calD)$ inside $\G_m^2 \times \calC $ is defined by an irreducible polynomial $A \in \kappa(\calC)[X,Y]$, which, since we are assuming $Z$ is not empty, is not a monomial. Furthermore, we can enlarge $S$ such that each coefficient of $A$ is an $S$-unit (and the extension depends only on $H_{\calC}(A)$ and $\deg A$). The pullback $\pi^*A$ is a regular function on $\calX \setminus \calD$. We claim that the degree of $\sigma^*(Z) \setminus S$ is bounded by the sum of the orders of $A$ evaluated at two $S$-units.

    Note that we can assume that $Z$ is a Cartier divisor: it is always $\Q$-Cartier (see Section \ref{sec:ram_pos}) and therefore there exists a positive integer $\ell$ such that $\ell Z \sim Z'$ with $Z'$ Cartier. Then, since $\sigma^*(Z) = \frac{1}{\ell} \sigma^*(Z')$ and $\ell$ is independent of $\sigma$, we can reduce to the case in which $Z$ itself is a Cartier divisor. Therefore, since $\calX$ is a normal variety, given an (affine) open $V \subset \calX$, $Z$ is locally defined in $V$ by an equation $f_V = 0$. The fact that $Z$ is contained in the ramification locus of $\pi$ implies that $f_V^2$ divides $\pi^*A$ as elements of the local ring $\calO_{V,Z}$. In particular, we can write locally $\pi^*A = f_V^2 g$ for a regular function $g$ in $V$.
    Consider now a point $P \in \calC \setminus S$ such that $\sigma(P) \in Z \cap V$; the contribution of $P$ to the divisor $\sigma^*(Z) \setminus S$ is $\ord_P(f_V \circ \sigma)$. On the other hand, since $\pi^* A = f_V^2 g$ in $\calO_{V,Z}$ for a regular function $g$, we have
    \[
      \ord_P(A(u,v)) \geq 2 \ord_P(f_V \circ \sigma),
    \]
where $\pi \circ \sigma = (u,v)$ for two $S$-units $u,v$ of $\calC$. Hence, we can bound the degree of $\sigma^* (Z) \setminus S$ by estimating the number of multiple zeros of the polynomial $A$ evaluated at the two $S$-units $u$ and $v$. Formally:
    \begin{equation} 
      \deg (\sigma^*(Z) \setminus S) = \sum_{\sigma(P) \in Z} \ord_P (f_v \circ \sigma) \leq \sum_{\sigma(P) \in Z} \left ( \ord_P ( A(u,v)) - 1 \right ).
      \label{eq:ram}
    \end{equation}
    Since we are estimating the degree of $\sigma^*(Z)$ restricted to $\calC \setminus S$, the sum on the right includes only $P \in \calC \setminus S$. To estimate this sum we want to apply Theorem \ref{th:1.2}. Fix $\varepsilon > 0$ and assume first that $u,v$ are multiplicatively independent, i.e. they do not satisfy any multiplicative relation of the form $u^r v^s = \gamma$ for a suitable pair of nonzero integers $(r,s) \in \Z^2 \setminus \{(0,0)\}$ and $\gamma$ an algebraic function of the coefficients of $A$. We can apply Theorem \ref{th:1.2} to obtain directly from \eqref{eq:ram} a bound of the form 
    \begin{equation}\label{eq:5}
    \deg (\sigma^*(Z) \setminus S) \leq \varepsilon H,
    \end{equation}
    provided that $H$ is bigger than $\Theta_1 \max\{1, \chi_S(\mathcal C)\}$, where $\Theta_1(\deg A, H_{\calC}(A),\varepsilon)$ is the explicit constant of Theorem \ref{th:1.2}. This proves the conclusion in this case.\medskip

    Assume on the contrary that there exist $(r,s) \in \Z^2 \setminus \{(0,0)\}$ and a rational function $\gamma \in \kappa(\calC)$ which is an algebraic function of the coefficients of $A$, such that $u^r v^s = \gamma$. Then Theorem \ref{th:1.2} implies that, either the same conclusion as in \eqref{eq:5} holds, or there is a bound of the form $\max \{ |r|, |s| \} \leq \Theta_2$ for a constant $\Theta_2$ that depends only on $\deg A$ and $\varepsilon$. In this latter case, the curve $(\pi \circ \sigma)(\calC)$ is a curve of degree bounded by $2 \Theta_2 \, H_\calC(\gamma)$. Therefore, the intersection between $(\pi \circ \sigma)(\calC)$ and $\pi(Z)$ is bounded by $\left( 2 \Theta_2\, H_\calC(\gamma)\right) \deg A$.

    In this case we obtain that
\begin{equation*}
	\deg (\sigma^*(Z) \setminus S) \leq 2 \Theta_2 H_\calC(\gamma) (\deg A) (\deg \pi) \leq \varepsilon H,
\end{equation*}
provided that $H \geq \dfrac{2 \Theta_2 H_\calC(\gamma) (\deg A) (\deg \pi)}{\varepsilon}$. Hence setting 
\[
	C = \max\left\{ \Theta_1, \dfrac{2 \Theta_2 H_\calC(\gamma) (\deg A) (\deg \pi)}{\varepsilon} \right\}
\]
gives the desired bound.
\end{proof}

%
%
%
%

\section{Positivity of the ramification divisor}
\label{sec:ram_pos}

In Proposition \ref{prop:ram} we estimated the contribution of the ramification divisor $Z$ to the degree of $\sigma(\calC)$. In this section we study positivity properties of this divisor. We first show that $Z$ is linearly equivalent to the divisor $K_{\calX/\calC} + \calD$.

\begin{lemma}[{{see also \cite[Lemma 1]{CZGm} in the split case}}]
  Let $\calX,\calD, Z$ be defined as in Section \ref{sec:3folds}. Then
  \[
    Z \sim K_{\calX/\calC} + \calD.
  \]
  In particular if $K_{\calX/\calC} + \calD$ is big then $Z$ is a big divisor on $\calX$. 
  \label{lem:Zbig}
  \begin{proof}
   By assumption $(\calX,\calD)$ is log canonical, and therefore $\calX$ is normal. Since we are assuming that $\calD$ is a Cartier divisor, this implies that $K_\calX$ is $\Q$-Cartier. Recall that, for the fibration $\rho: \calX \to \calC$, the canonical divisor of $\calX$ verifies
    \[
      K_\calX = K_{\calX/\calC} + \rho^* (K_\calC).
    \]
    At the same time, the canonical divisor of $\PP^2 \times \calC$ is $K_{\PP^2} \boxplus K_\calC := pr_1^*(K_{\PP^2}) + pr_2^*(K_\calC)$, and therefore the Riemann-Hurwitz formula for $\pi: \calX \to \PP^2 \times \calC$ implies that
    \begin{align*}
      K_\calX &= \pi^* (K_{\PP^2} \boxplus K_\calC) + Ram \\
      &= \pi^*pr_1^* (K_{\PP^2}) + \pi^*pr_2^*(K_\calC) + Ram \\
      &= \pi^*pr_1^* (K_{\PP^2}) + \rho^* (K_\calC) + Ram, 
    \end{align*}
    where $Ram$ is the ramification divisor of $\pi$ and we used that $\rho = pr_2 \circ \pi$. Therefore we obtain that 
    \[
      K_{\calX/\calC} - \pi^*pr_1^* (K_{\PP^2}) = Ram.
    \]
    Since $K_{\PP^2} \sim - (\PP^2 \setminus \G_m^2)$, the pullback $\pi^*pr_1^*(K_{\PP^2})$ is linearly equivalent to $-(\calD + Ram_\calD)$, where $Ram_\calD$ is the ramification coming from the support of $\calD$. Similarly, $Ram = Z + Ram_\calD$ and therefore $K_{\calX/\calC} + \calD \sim Z$, as wanted.
  \end{proof}
\end{lemma}

Lemma \ref{lem:Zbig} was used in the split case to show that, if $X$ is of log general type, then $Z$ is a big divisor. However, in our situation, Lemma \ref{lem:Zbig} only implies that the restriction of $Z$ to the generic fiber is big, which is sometimes referred as $Z$ is \emph{relatively big} or $\rho$-big. This follows from the fact that 
\[
	(K_{\calX/\calC} + \calD) \restriction_{\calX_\eta} \cong K_{\calX_\eta} + \calD_\eta,
\]
and we are assuming that the generic fiber has the property that the log canonical divisor is big. This in general does not imply that $Z$ is a big divisor (consider for example a trivial family over the base).  
Note however that in certain cases $Z$ is already positive, as shown in the following remark.

\begin{remark}\label{rmk:p2}
  One important example of applications of the main theorem is for complements of divisors in $\PP^2$. In this case, $\calX = \PP^2 \times \calC$ and $D$ is an effective member of the linear system $\lvert \calO_{\PP^2}(m) \boxtimes \calL \rvert := \lvert pr_1^* \calO_{\PP^2}(m) \otimes pr_2^* \calL \rvert$. The condition that the fibers of $\rho = pr_2$ are of log general type implies that $m \geq 4$ and the fact that the fibration is not isotrivial implies that $\deg_\calC \calL > 0$. Then a direct computation shows that $K_{\calX/\calC} \sim \calO_\PP^2(-3) \boxtimes \calO_\calC$ and therefore
  \[
  \calO_{\calX} ( K_{\calX/\calC} + D ) = (\calO_{\PP^2}(-3) \boxtimes \calO_\calC) \otimes (\calO_{\PP^2}(m) \boxtimes \calL )= \calO_{\PP^2}(m-3) \boxtimes \calL,
  \]
  which is ample since $m \ge 4$ and $\deg_\calC \calL >0$. This shows that $Z = K_{\calX/\calC} + D$ is big in this case.
\end{remark}

\noindent To obtain positivity properties of the divisor $Z$ it is sufficient to add to it the pullback of a divisor from the base curve. More specifically, it is enough to add to $Z$ the pullback of the divisor $K_\calC + 2 S$ on the curve $\calC$. We note that, by hypothesis, $\#S \geq 2$ and therefore the divisor $K_\calC + 2S$ is very ample on $\calC$. 

\begin{proposition}\label{prop:Z+e}
	The divisor $Z + \rho^* (K_\calC + 2S)$ is big. Moreover there exists $m>0$ that depends only on the generic fibers of $(\calX,\calD)$ and $Z$ such that $m(Z + \rho^*(K_\calC + 2S)) \sim  H + E$ with $H$ a very ample divisor and $E$ an effective divisor on $\calX$.
\end{proposition}
\begin{proof}
The argument is somehow standard; we reproduce it here for completeness. Since $Z$ is $\rho$-big, there exists a $\rho$-ample Cartier divisor $H_1$ and an effective Cartier divisor $E$ on $\calX$ such that $\rho_*\calO_\calX(E) \neq 0$ and $mZ \sim H_1 + E$ for some $m \geq 1$ that depends on $X$ and the generic fiber of $Z$ (see \cite[Proposition 1.6.33]{DFEM}). Moreover, by \cite[\href{https://stacks.math.columbia.edu/tag/01VU}{Lemma 01VU}]{stacks-project} we can assume that $H_1$ is in fact very ample up to replacing $m$ by a large enough multiple. By our assumption on $S$ we see that $K_\calC +2S$ is very ample on $\calC$, since its degree satisfies $\deg K_\calC + 2S \geq 2g + 1$. Then, by \cite[Proposition 1.6.15(i)]{DFEM}, the divisor $H_1 + \rho^*(K_\calC + 2S)$ is very ample on $\calX$, hence
  \[
    m(Z + \rho^*(K_\calC + 2S)) \sim H_1 + \rho^*\left( m(K_\calC + 2S) \right) + E.
  \]
  This shows that $Z + \rho^*(K_\calC + 2S)$ is big and the desired equivalence holds with $H = H_1 + \rho^* m(K_\calC + 2S)$.
\end{proof}

Proposition \ref{prop:Z+e} will be used in the next section to bound the degree of a section by estimating the intersection with the divisor $Z$. 
We end this section with a result that allows a more direct control on the positivity of $Z$ when the family has nice properties coming from the theory of stable pairs.  

When the log canonical divisor of the generic fiber of $\rho$ is ample, we can choose a model of $(X,D)$ over $\calC$ that is a stable family in the sense of Koll\'ar (see for example \cite[Definition 2.7]{fpt}). In this situation we show that, when $\calD$ is an ample divisor on $\calX$, the divisor $K_{\calX/\calC} + \calD$, and therefore $Z$, is indeed big. This follows from a result of Kov\'acs and Patakfalvi in \cite{kp}.

\begin{proposition}
  \label{prop:Zbig}
  Let $(X,D) \to B$ be a stable family such that the generic fiber has log canonical singularities, the base $B$ is a nonsingular projective curve and $D$ is an ample divisor on $X$. Then, $K_{X/B} + D$ is big.
\end{proposition}
\begin{proof}
  Let $E$ be an irreducible divisor in the linear system $nD$ and consider the pair $\left (X,\frac{1}{n}E \right ) \to B$: it is a stable family where the generic fiber has now klt singularities (see \cite[Definition 2.8]{singmmp}). The fact that $D$ is ample and the base is one dimensional implies that the variation of the family is maximal and therefore \cite[Corollary 7.3]{kp} implies that $K_{X/Y} + \frac{1}{n}E$ is big, finishing the proof.
\end{proof}

Proposition \ref{prop:Zbig} has the advantage of reducing the problem of testing whether the divisor $Z$ is big to an ampleness condition on the model of the divisor $\calD$, a condition that is easier to check in the applications.

%
%
%
\section{Geometric Lang-Vojta for ramified covers of $\G_m^2$}
\label{sec:main}
%
%
%
In this section we apply Proposition \ref{prop:ram} together with Proposition \ref{prop:Z+e} to deduce Theorem \ref{th:main}. This extends explicitly \cite[Theorem 2]{CZGm} to the case of non-isotrivial pairs.

\begin{thmyy}
  \label{th:2}
  Let $\rho: (\calX,\calD) \to \calC$ be a fibered threefold, let $Z$ be the closure of the ramification divisor of the finite map $\pi: \calX \setminus \calD \to \G_m^2 \times \calC$ as defined in Section \ref{sec:3folds}, and let $Z_\eta$ be its generic fiber.
 
  Then, there exists a constant $C = C(Z_\eta, \deg \pi)$ such that, for every section $\sigma: \calC \to \calX$ with $\supp (\sigma^*\calD) \subseteq S$, the following holds:
  \[
    \deg \sigma(\calC) \leq C \cdot \max\{ 1, \chi_S(\calC) \}.
  \]
\end{thmyy}
\begin{proof}
	By Proposition \ref{prop:Z+e}, the divisor $Z+\rho^*(K_\calC + 2S)$ is a big $\Q$-Cartier divisor and there exist an integer $m$ indipendent of $\calC$, an effective divisor $E$ (indipendent of $\calC$) and a very ample divisor $H$ on $\calX$ such that $m(Z + \rho^*(K_\calC + 2S)) \sim H + E$. We can estimate the degree of $\sigma(\calC)$ by estimating $\sigma(\calC) \cdot H$.
	Recall that the constant $C$ implicitly depends on the projective embedding $\calX \rightarrow \PP^n$, and therefore on $H$. 
  
	First we notice that, if $\sigma(\calC)$ is contained in the support of $E$, then we obtain a bound of the desired form immediately. Indeed, assuming $E$ is irreducible by arguing on every irreducible component, since $\sigma(\calC) \subset E$, the restriction $\rho\vert_E$ is a flat dominant map where now $E$ has relative dimension 1, i.e. $E$ is a curve over the function field $\kappa(\calC)$. The generic fiber of the image $\pi(E)$ intersects $\pi(D)$ in at least three points, since our assumptions imply that the generic fiber of $\pi(Z)$ is disjoint from the singular points of $\PP^2 \setminus \G_m^2$. This shows that $E/\kappa(\calC)$ is a curve of log general type hence Siegel's Theorem applies (e.g. in the form of \cite{Lang_int}): in particular, either there are only finitely many integral sections $\sigma \in E(\kappa(\calC))$ or $E$ is isotrivial. In both cases we obtain directly that $\deg(\sigma(\calC))$ is bounded by $c_1 = c_1(Z_\eta, \deg \pi)$.

  We are therefore left with the case in which $\sigma(\calC)$ is not contained in the support of $E$. Then,
  \begin{equation}
	  \frac{1}{m} \deg \sigma(\calC) = \frac{1}{m} \sigma(\calC) \cdot H \leq \frac{1}{m} \sigma(\calC) \cdot (H + E) = \sigma(\calC) \cdot Z + \sigma(\calC) \cdot \rho^*(K_\calC + 2S).
    \label{eq:big}
  \end{equation}

  Since $\sigma$ is a section we get that $\sigma(\calC)\cdot \rho^*(K_\calC + 2S)$ is equal to the degree of $(K_\calC + 2S)$, hence
  \begin{equation}
    \sigma(\calC)\cdot \rho^*(K_\calC + 2S) \leq 2 \cdot \max \{ 1, \chi_S(\calC) \}.
  \end{equation}
  Therefore it is enough to bound the intersection $\sigma(\calC) \cdot Z$.
  
We view $\G_m^2$ embedded in $\PP^2$ as the complement of the divisor $UVW=0$, where $U, V$ and $W$ are the homogeneous coordinates in $\PP^2$. 
Since $\pi(Z_{\eta})$ avoids the singular points of $\PP^2 \setminus \G_m^2$,  if we consider the morphism $\pi \circ \sigma: \calC \to \PP^2 \times \calC$ restricted to the complement of $S$ (which by abuse of notation we still denote by $\pi \circ \sigma$), there exist two $S$-units $u,v$ such that $\pi \circ \sigma$ is given by $( (u:v:1),\ \id_\calC)$. 
  
  Let $f(U,V,W) = 0$ be an equation for $\pi(Z)$, which is a polynomial with coefficients in $\kappa(\calC)$. Then, we can estimate the intersection $\sigma(\calC) \cdot Z$ by counting the number of zeros of $f$ evaluated at $(u,v,1)$. Therefore \eqref{eq:big} can be written as
  \begin{equation}
	  \sum_{P \in \calC} \max\{ 0, \ord_P f(u,v,1)\} \geq \frac{1}{m \deg \pi} \max\{ H(u), H(v) \} - \frac{2}{\deg \pi}\max \{ 1, \chi_S(\calC) \}.
    \label{eq:HB1}
  \end{equation}
 
On the other hand, if at least one of the heights of $u$ and $v$ is larger than the constant appearing in Proposition \ref{prop:ram}, we have that for every $\varepsilon > 0$ the following equality holds:
\begin{equation}
    \sum_{P \in \calC \setminus S} \max\{ 0, \ord_P f(u,v,1) \} \le \varepsilon \max\{ H(u), H(v) \}.
  \label{eq:HB2}
\end{equation}

Up to a finite extension of $S$, independent of $u$ and $v$, we can assume that the coefficients of $f$ are $S$-units, and therefore $f(u,v,1)$ is an $S$-integer. Given the hypothesis on the intersection between $\pi(Z_{\eta})$ and the boundary $\PP^2 \setminus \G_m^2$, the constant term $f(0,0,1) = \mu$ is a nonzero $S$-unit (independent of $u$ and $v$). We write $f(u,v,1)$ as sum of monomials using the following notation

  \[
    w_1 := -f(u,v,1) \qquad \qquad f(u,v,1) := \sum_{j = 2}^n w_j.
  \]
  We get $w_1 + \cdots + w_n = 0$, where $w_n = \mu$. If a proper subsum vanishes, this gives an equation of the form $g(u,v) = 0$ which implies that $\deg\pi(\sigma(\mathcal C))$ is bounded by $\deg g \le \deg f$ thus finishing the proof. In particular, this might happen when $u$ and $v$ are multiplicatively dependent, and one obtains a bound on the exponents of the multiplicative relation.

  Therefore, we can assume that no proper subsum of $w_1 + \cdots + w_n = 0$ vanishes, so that we can apply \cite[Theorem B]{Brownawell1986} to obtain
  \begin{equation}
	  H(w_1: \cdots: w_n) \leq \gamma_n \max \{ 0, 2g -2\} + \sum_{P \in \calC} (\gamma_n - \gamma_{m_P}),
    \label{eq:BM}
  \end{equation}
  where
  \[
    \gamma_0 = 0, \qquad \gamma_{\ell} = \frac{1}{2} (\ell-1)(\ell-2) \quad \forall \ell \ge 1,
  \]
  and $m_{P}$ denotes the number of $w_i$s that are units at $P$. The right hand side of \eqref{eq:BM} can be bounded from above obtaining
  \begin{equation}
	  H(w_1: \cdots: w_n) \leq \gamma_n \max\{ 0, 2g-2 \} +\sum_{P \in S} \gamma_n + \sum_{P\in \calC \setminus S} (\gamma_n - \gamma_{m_P}).
    \label{eq:BM2}
  \end{equation}

  Notice that $\gamma_{m_P} = \gamma_n$ for all but finitely many $P \in \calC \setminus S$. In this finite set, which is the set of zeros of $f(u,v,1)$ outside $S$, we have $\gamma_{m_P} = \gamma_{n-1}$. Therefore we can rewrite equation \eqref{eq:BM2} as

  \begin{equation*}
	  H(w_1:\cdots: w_n) \leq 2\gamma_n \max\{ 1, 2g - 2 + \# S \}+ (n-2) \sum_{P \in \calC \setminus S} \max \{ 0, \ord_P f(u,v,1) \}.
  \end{equation*}
  
  Using \eqref{eq:HB2} we have that, for every $\varepsilon > 0$,
  \begin{equation}
      H(w_1:\cdots: w_n) \leq (n-1)(n-2) \max\{ 1, 2g - 2 + \# S \} + (n-2) \varepsilon \max \{ H(u), H(v) \}.
    \label{eq:BM4}
  \end{equation}


Since $w_1=-f(u,v,1)$ and $w_n=\mu$, we can bound from below the projective height as
  \begin{equation}
	  H(w_1:\cdots:w_n) \geq H(f(u,v,1) : \mu) \geq \sum_{P \in \calC} \max \{0, \ord_P(f(u,v,1)) \} - H(\mu).
    \label{eq:BM6}
  \end{equation} 
Combining \eqref{eq:HB1} with \eqref{eq:BM6}, we have
  \begin{equation}
    H(w_1:\cdots:w_n) \geq \frac{1}{m \deg \pi} \max \{ H(u), H(v) \}- \frac{2}{\deg \pi}\max \{ 1, \chi_S(\calC) \} - H(\mu).
    \label{eq:BM7}
  \end{equation}
  Using \eqref{eq:BM4} with $\varepsilon = \dfrac{1}{ 2m(n-2) \deg \pi} $ together with \eqref{eq:BM7} we obtain
  \begin{equation} \label{eq:height_uv}
     \max\{ H(u), H(v) \} \leq c_2 \max \left \{1, \chi_S(\mathcal C)\right \},   
\end{equation}
where
\[
c_2:=\max \left \{2m \deg \pi \left ( (n-1)(n-2) + \frac{2}{\deg \pi} + H(\mu) \right ), c_3 \right \}.
\]
and $c_3$ is the constant appearing in Proposition \ref{prop:ram}, which depends only on $Z_\eta$ and on $\deg \pi$. 

Finally, we obtained that either the degree of $\sigma(\mathcal C)$ is directly bounded by a constant which depends only on $Z_\eta$ and on the degree of $\pi$, or the maximum of the heights of $u$ and $v$ is bounded by \eqref{eq:height_uv}, which implies a bound for the degree of the image $\sigma(\calC)$.
\end{proof}

\section{Applications}
\label{sec:appl}

In this section we give an explicit application of Theorem \ref{th:2}, which generalizes \cite[Theorem 1]{CZGm} and \cite[Theorem 1.3]{TurchetTrans}.

\begin{thmyy}
  \label{th:exP2}
  Let $\calD$ be a non-isotrivial stable flat family of divisors in $\PP^2$ over a smooth integral curve $\calC$. Assume that there exists a finite set of points $S \subset \calC$ such that, for every $P \in \calC \setminus S$, the fiber $\calD_P$ of $\calD$ over $P$ has simple normal crossings singularities and $r \geq 3$ components of total degree $d \geq 4$. Then, there exists a constant $C$, depending only on the generic fiber of $\calD$, such that for every section $\sigma: \calC \to \PP^2 \times \calC$ verifying $\supp (\sigma^*\calD) \subseteq S$ the following holds:
  \[
    \deg \sigma(\calC) \leq C \cdot \max\{ 1, \chi_S(\calC) \}.
  \]
\end{thmyy}
\begin{proof}
  By Remark \ref{rmk:p2}, $\calD$ is an ample (Cartier) divisor on $\PP^2 \times \calC$, and therefore we can assume that $\calD \in \lvert \calO_{\PP^2}(d) \boxtimes \calL \rvert$ for an ample $\calL \in \Pic(\calC)$ and $d \geq 4$. Let $0 \in \calC$ be the generic point and let $D^0$ be the generic fiber of the family $\calD \to \calC$. The hypotheses ensure that we can write $D^0 = D_1 + D_2 + \dots + D_r$, with $r \geq 3$. Let $g_{1}$ be an equation of $D_1$, $g_{2}$ be an equation of $D_2$, and $g_{3}$ be an equation of $D_3 + \dots + D_r$, of degree respectively $d_1, d_2$ and $d_3$; clearly $d_1 + d_2 + d_3 = d$. Let us consider the map
  \[
    \pi_0: \PP^2 \to \PP^2 \qquad \pi_0([x_0:x_1:x_2]) = [g_1^{d_2d_3}([x_0:x_1:x_2]): g_2^{d_1d_3}([x_0:x_1:x_2]): g_3^{d_1d_2}([x_0:x_1:x_2])],
  \]
  which is clearly a finite map and it is defined everywhere since $D^0$ has simple normal crossings singularities. Moreover $\pi_0(D^0) = \PP^2 \setminus \G_m^2$ and therefore $\pi_0$ restricts to a finite cover of $\G_m^2$ on $\PP^2 \setminus D^0$. The same argument as in \cite[Theorem 1]{CZGm} implies that if $E$ is an irreducible component of the ramification divisor of $\pi_0$ such that $\pi_0(E)$ passes through a singular point of $\PP^2 \setminus \G_m^2$ then $E$ lies in the support of $D^0$. 
  Therefore, we can extend $\pi_0$ to a generically finite map $\pi: \PP^2 \times \calC \dashrightarrow \PP^2 \times \calC$ which, after possibly resolving indeterminacy of the map, fits into the diagram 
  \[
    \xymatrix{ (\PP^2 \times \calC,\calD) \ar[d]_{pr_2}\ar[r]^\pi & \PP^2 \times \calC \ar[dl]_{pr_2} \\ \calC \ar@/^2.0pc/[u]^\sigma & }
  \]
  By the above discussion, the ramification divisor of $\pi$ has the property that every irreducible component $E$ such that $pr_1(\pi(E))$ passes through one of the singular points of $\PP^2 \setminus \G_m^2$ is contained in the support of $\calD$; in particular, $pr_2\left(\pi(Z)\cap\left( (\PP^2 \setminus \G_m^2)\times \calC\right)\right)$ is contained in $S$, where $Z$ is the ramification divisor of $\pi \restriction_{(\PP^2 \times \calC) \setminus \calD}$.
  Therefore, we can apply Theorem \ref{th:2} whose conclusion implies the theorem.
\end{proof}

\noindent We finally give an explicit example where Theorem \ref{th:exP2} applies. 

\begin{example}
Consider the $1$ parameter family of reducible plane quartics given by the vanishing of 
  \[
	  (y_0x_0)(y_0x_1)(y_0^2x_2^2-y_0^2x_1^2-y_1^2x_1x_0-y_0^2x_0^2),
  \]
  in $\PP^2 \times \PP^1$, and let $\calD \in \lvert \calO_{\PP^2}(4) \boxtimes \calO_{\PP^1}(4) \rvert$ be the associated divisor. Let $S$ be the finite set of points of $\PP^1$ for which the specialized divisor has no normal crossings singularities or does not have 3 irreducible components; note that $S$ contains at least two points, e.g. $[0:1]$ and $[1:\sqrt{2}]$. We define as in Theorem \ref{th:exP2}
  \[
	  \pi_0([x_0 : x_1 : x_2]) = [(y_0x_0)^2 :  (y_0x_1)^2 :  y_0^2x_2^2 - y_0^2x_1^2 -  y_1^2x_1 x_0 - y_0^2x_0^2],
  \]
  which is clearly a finite map and well defined outside of the generic fiber $D^0$ of $\calD$. In this case, one can compute explicitly the ramification of the finite map $\pi_0$ which is given by $(y_0x_0)(y_0x_1)(y_0^2x_2) = 0$; in particular, it has 3 irreducible components. Two of them, namely $y_0x_0 = 0$ and $y_0x_1 = 0$, are contained in the support of $D^0$, while $y_0^2x_2 = 0$ corresponds to the generic fiber of $Z$ in our previous notation. We see directly that $Z$ is an ample divisor lying in the linear system $\lvert \calO_{\PP^2}(1) \boxtimes \calO_{\PP^1}(2) \rvert$ and moreover $Z \sim K_{\PP^2 \times \PP^1} + \calD$. Clearly, since $S$ contains the point $[0:1]$, one sees directly that $pr_2\left(\pi(Z)\cap\left( (\PP^2 \setminus \G_m^2)\times \calC\right)\right)$ is contained in $S$, and therefore we can apply Theorem \ref{th:exP2} obtaining a bound on the degree of the image of sections $\sigma: \PP^1 \to \PP^2 \times \PP^1$ such that $\supp (\sigma^*\calD)$ is contained in $S$.
\end{example}

 %
 %
 %
 %

%


\bibliographystyle{alpha}	
\bibliography{gm}
	

\end{document}